\documentclass[12pt]{amsart}
\usepackage[margin=3cm]{geometry}
\usepackage{tikz}
\usetikzlibrary{arrows.meta}
\usepackage{overpic}
\usepackage{graphicx} 
\usepackage{mathtools}
\usepackage{pgfplots}
\pgfplotsset{compat=1.18} 
\usetikzlibrary{pgfplots.fillbetween}
\usepackage{url}
\usepackage{xcolor}
\usepackage[
  colorlinks=true,
  linkcolor=blue,
  citecolor=blue,
  urlcolor=blue
]{hyperref}

\usepackage{amsmath, amsthm, amsfonts,stmaryrd,amssymb,algorithm, algpseudocode}

\usepackage[normalem]{ulem}

\title[Semi-discrete convex order and Laguerre tessellations]{Semi-discrete convex order and \\Laguerre tessellation fitting}

\let\mc=\mathcal
\let\mr=\mathrm

\def\ed{\mathrm{d}}

\theoremstyle{plain}
\newtheorem{theorem}{Theorem}[section]
\newtheorem{remark}[theorem]{Remark}
\newtheorem{lemma}[theorem]{Lemma}

\newtheorem{proposition}[theorem]{Proposition}

\newtheorem{question}{Question}

\theoremstyle{definition}
\newtheorem{definition}[theorem]{Definition}
\newtheorem{example}[theorem]{Example}

\date{}
\author[D.\ P.\ Bourne]{David P.\ Bourne}
\address{David P.\ Bourne, Maxwell Institute for Mathematical Sciences and Department of Mathematics, Heriot-Watt University, Edinburgh, UK}
\email{d.bourne@hw.ac.uk}

\author[T.\ O. Gallou\"et]{Thomas O.\ Gallou\"et}
\address{Thomas O. Gallou\"et, Université Paris-Saclay, CNRS, Inria, Laboratoire de mathématiques d’Orsay, ParMA, 91405, Orsay, France}
\email{thomas.gallouet@inria.fr}

\author[Q.\ M\'erigot]{Quentin M\'erigot}
\address{Quentin M\'erigot. Universit\'e Paris-Saclay, CNRS, Inria, Laboratoire de math\'ematiques d’Orsay,
91405, Orsay, France / DMA, \'Ecole normale superieure, Université PSL, CNRS, 75005 Paris,
France / Institut universitaire de France (IUF)}
\email{quentin.merigot@universite-paris-saclay.fr}

\author[A. Natale]{Andrea Natale}
\address{Andrea Natale, Université Paris-Saclay, CNRS, Inria, Laboratoire de mathématiques d’Orsay, ParMA, 91405, Orsay, France}
\email{andrea.natale@inria.fr}

\begin{document}

\begin{abstract}  Laguerre tessellations offer an efficient  way to parameterize a large class of convex partitions of Euclidean space using only a set of points and scalar weights. For this reason, they have become popular in computational geometry, imaging and numerical analysis, both as a modeling and a discretization tool. In this paper we study the problem of reconstructing a Laguerre tessellation with prescribed cell volumes from the barycenters of its cells. 
We establish a geometric interpretation of this problem in terms of the set of discrete measures dominated in convex order by an absolutely continuous measure. In particular,
we show that the reconstruction problem can be solved approximately by computing a Wasserstein projection onto this set. More generally,  our method can also be applied to fit a Laguerre tessellation to an arbitrary set of barycenters. 
We give a concrete application of this in materials science, of fitting a Laguerre tessellation to an electron backscatter diffraction (EBSD) image of a steel.
\end{abstract}

\maketitle

\section{Introduction}

Laguerre tessellations, also known as weighted Voronoi tessellations, power diagrams \cite{aurenhammer1987power}, affine partitions \cite{akopyan2012kadets}, or regular partitions \cite{leon2018spaces}, are partitions of $\mathbb{R}^d$ into convex cells that generalize Voronoi tessellations. Specifically, 
given a vector of distinct generators $Y =(y_1,\ldots,y_N)\in (\mathbb{R}^d)^N$, i.e., $y_i\neq y_j$ for all $i \neq j$,
and a vector of weights $\phi=(\phi_1,\ldots,\phi_N)\in\mathbb{R}^N$, the associated Laguerre tessellation is the partition of $\mathbb{R}^d$ defined by the cells
\[
L_i(\phi,Y) \coloneqq\{ x\in\mathbb{R}^d \,;\, \langle x, y_i\rangle-\phi_i \geq  \langle x,y_j\rangle -\phi_j ~ \forall i\neq j\}.
\]
Setting $\phi_i = |y_i|^2/2$, the definition above yields the classical Voronoi tessellation associated with $Y$.
In general, Laguerre tessellations form a strict subset of the space of all convex 
tessellations of $\mathbb{R}^d$. However, it turns out that for $d\geq3$ all convex  
tessellations  satisfying some mild geometric requirements (specifically, ``simple'' tessellations) are always Laguerre \cite[Theorem 4]{aurenhammer1987criterion}.

Because of their versatility, Laguerre tessellations have become an established modeling tool in various fields, including biology \cite{bock2010generalized}, optics \cite{merigot2021mirrors}, and  materials science \cite{BKRS20} to name a few. 
In particular, they are often used to describe the
microstructure of materials such as metals or foams. A common problem in this context 
consists in computing a Laguerre tessellation with prescribed geometrical properties coming
from experimental observations \cite{lyckegaard2011use,petrich2019reconstruction,BKRS20,gritzmann2026constrained}. Motivated by these applications, in this work we will study the following problem: 

\begin{question}\label{ques:reconstruction}
{Can one recover a set of generators and weights associated with a given Laguerre tessellation provided only the volumes and the barycenters of its cells?}
\end{question}

We refer to Figure \ref{fig:example} for an illustration of the problem, and to Section \ref{subsec: motivation} for more about the physical motivation. This problem was originally posed in \cite{bourne2024inverting}. In Theorem 4.5 of \cite{bourne2024inverting}, the authors show that a Laguerre tessellation is uniquely determined by the volumes and barycenters of its cells. However, the generators and weights are not unique in general.

The key idea of this work is to characterize the barycenters of Laguerre cells with prescribed volumes in terms of convex order between probability measures; see Section \ref{sec:convexoreder}. 
This establishes a direct connection between the reconstruction problem and semi-discrete optimal transport. 
We leverage this connection to propose an algorithm that approximately solves the reconstruction problem. Specifically, our method consists in computing a Wasserstein projection onto the set of measures in convex order with a given absolutely continuous measure. Importantly, the proposed approach can be used more generally on experimental data which does not necessarily arise from a Laguerre tessellation.

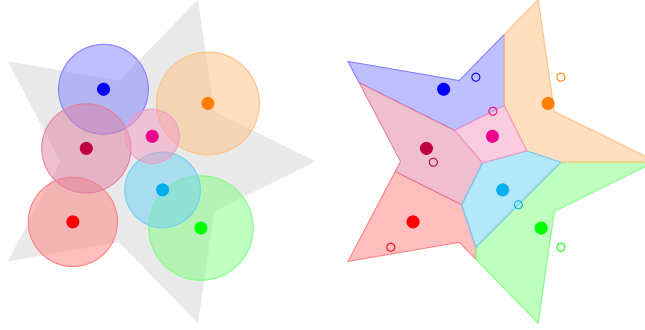
\begin{figure}
\begin{center}
\begin{tikzpicture}[scale=1.5]
\begin{scope}[shift={(0,0)}]
%\filldraw[lightgray!35, draw=lightgray!60, line width=0.4pt] 
\fill[lightgray!35] (2.0857142857142854,0.5071428571428571) -- (1.1924770314954962,0.9479817963622119) -- (1.0492397772767068,1.9337276315855874) -- (0.35395153993307515,1.2204352443642223) -- (-0.6278112058481353,1.388820735581567) -- (-0.16428571428571437,0.5071428571428572) -- (-0.6278112058481358,-0.3745350212958525) -- (0.35395153993307493,-0.206149530078508) -- (1.0492397772767066,-0.9194419172998733) -- (1.1924770314954962,0.0663039179235021) -- (2.0857142857142854,0.5071428571428571);
\filldraw[red!50, fill opacity=0.5] (-0.05417029710919895,-0.0262266544244478) circle (0.3935200853400884);
\filldraw[green!50, fill opacity=0.5] (1.075714284094312,-0.08067099477494923) circle (0.46074070914405196);
\filldraw[blue!50, fill opacity=0.5] (0.21607083656718973,1.14373038930874) circle (0.39597833432162);
\filldraw[orange!50, fill opacity=0.5] (1.1370750484492522,1.018708197462087) circle (0.45348969985873216);
\filldraw[purple!50, fill opacity=0.5] (0.06427467683735659,0.6218073100703538) circle (0.3932277518986199);
\filldraw[cyan!50, fill opacity=0.5] (0.7369601095872323,0.2560149197075473) circle (0.33362217097327956);
\filldraw[magenta!50, fill opacity=0.5] (0.6461480937824722,0.7282064329166794) circle (0.23854102094887197);
\filldraw[red] (-0.05417029710919895,-0.0262266544244478) circle (1.5pt);
\filldraw[green] (1.075714284094312,-0.08067099477494923) circle (1.5pt);
\filldraw[blue] (0.21607083656718973,1.14373038930874) circle (1.5pt);
\filldraw[orange] (1.1370750484492522,1.018708197462087) circle (1.5pt);
\filldraw[purple] (0.06427467683735659,0.6218073100703538) circle (1.5pt);
\filldraw[cyan] (0.7369601095872323,0.2560149197075473) circle (1.5pt);
\filldraw[magenta] (0.6461480937824722,0.7282064329166794) circle (1.5pt);
\end{scope}
\begin{scope}[shift={(3.0,0)}]
\filldraw[red!50, fill opacity=0.5] (0.375,0.125) -- (0.5,-0.25) -- (0.5,-0.35597984380443226) -- (0.35395153993307493,-0.206149530078508) -- (-0.6278112058481358,-0.3745350212958525) -- (-0.2111194804214762,0.4180597402107381) -- (0.375,0.125);
\filldraw[green!50, fill opacity=0.5] (0.5,-0.25) -- (1.25,0.5) -- (2.0712412763361643,0.5) -- (1.1924770314954962,0.0663039179235021) -- (1.0492397772767066,-0.9194419172998733) -- (0.5,-0.35597984380443226) -- (0.5,-0.25);
\filldraw[blue!50, fill opacity=0.5] (-0.5298054606168581,1.202402730308429) -- (-0.6278112058481353,1.388820735581567) -- (0.35395153993307515,1.2204352443642223) -- (0.75,1.626739186188084) -- (0.75,1.0) -- (0.3125,0.78125) -- (-0.5298054606168581,1.202402730308429);
\filldraw[orange!50, fill opacity=0.5] (0.95,0.6) -- (0.75,1.0) -- (0.75,1.626739186188084) -- (1.0492397772767068,1.9337276315855874) -- (1.1924770314954962,0.9479817963622119) -- (2.0857142857142854,0.5071428571428571) -- (2.0712412763361643,0.5) -- (1.25,0.5) -- (0.95,0.6);
\filldraw[purple!50, fill opacity=0.5] (-0.2111194804214762,0.4180597402107381) -- (-0.16428571428571437,0.5071428571428572) -- (-0.5298054606168581,1.202402730308429) -- (0.3125,0.78125) -- (0.5592105263157895,0.4934210526315789) -- (0.375,0.125) -- (-0.2111194804214762,0.4180597402107381);
\filldraw[cyan!50, fill opacity=0.5] (1.25,0.5) -- (0.5,-0.25) -- (0.375,0.125) -- (0.5592105263157895,0.4934210526315789) -- (0.95,0.6) -- (1.25,0.5);
\filldraw[magenta!50, fill opacity=0.5] (0.5592105263157895,0.4934210526315789) -- (0.3125,0.78125) -- (0.75,1.0) -- (0.95,0.6) -- (0.5592105263157895,0.4934210526315789);
\filldraw[red] (-0.05417029710919895,-0.0262266544244478) circle (1.5pt);
\filldraw[green] (1.075714284094312,-0.08067099477494923) circle (1.5pt);
\filldraw[blue] (0.21607083656718973,1.14373038930874) circle (1.5pt);
\filldraw[orange] (1.1370750484492522,1.018708197462087) circle (1.5pt);
\filldraw[purple] (0.06427467683735659,0.6218073100703538) circle (1.5pt);
\filldraw[cyan] (0.7369601095872323,0.2560149197075473) circle (1.5pt);
\filldraw[magenta] (0.6461480937824722,0.7282064329166794) circle (1.5pt);
\draw[red] (0.0*1.5-.25,0.0*1.5-.25) circle (1pt);
\draw[green] (1.0*1.5-.25,0.0*1.5-.25) circle (1pt);
\draw[blue] (0.5*1.5-.25,1.0*1.5-.25) circle (1pt);
\draw[orange] (1.0*1.5-.25,1.0*1.5-.25) circle (1pt);
\draw[purple] (0.25*1.5-.25,0.5*1.5-.25) circle (1pt);
\draw[cyan] (0.75*1.5-.25,0.25*1.5-.25) circle (1pt);
\draw[magenta] (0.6*1.5-.25,0.8*1.5-.25) circle (1pt);
\end{scope}
\end{tikzpicture}
\end{center}
\vspace{-3em}
\caption{We seek to reconstruct a Laguerre tessellation given its cells' barycenters (i.e., their centroids) and volumes (computed with respect to a given density; here constant on a star), represented by the disks on the left. The empty dots on the right are the unknown generators of the tessellation.} \label{fig:example}
\end{figure}

\subsection{Motivation from materials science.}
\label{subsec: motivation}
{Question \ref{ques:reconstruction} was inspired by \cite{bourne2024inverting} and the image reconstruction problem in \cite{petrich2019reconstruction}, which is about reconstructing the microstructure of polycrystalline materials from incomplete data, namely, far-field three-dimensional X-ray diffraction data (3DXRD). 
Polycrystalline materials, such as steel, are composed of \emph{grains}, which are regions where the atoms form a crystal lattice. 
The grains form a tessellation of the material; see Figure \ref{fig:fitting} (left), where the colors correspond to the grains.  The 3DXRD data consists of the volumes and barycenters of the grains, but the grains themselves are unknown. In \cite{petrich2019reconstruction} the authors reconstruct the grains by finding a Laguerre tessellation that approximately matches the measured volumes and barycenters, using stochastic optimization, where each cell in the Laguerre tessellation represents a grain. 
This reconstruction problem has also been studied for example by \cite{lyckegaard2011use,QR18,bourne2024inverting}. We propose a solution using convex optimization and apply it to  reconstruct an image of steel (see Section \ref{subsubsec:EBSD}).   
}

\subsection{Semi-discrete optimal transport}
\label{subsec: SDOT}
Laguerre tessellations with prescribed cell volumes arise naturally in the theory of optimal transport \cite{aurenhammer1998minkowski,MerigotThibertOT}. Let \(\mc{P}_1(\mathbb{R}^d)\) denote the set of probability measures with finite first moments, i.e., such that $\int |x| \ed \mu(x)<\infty$. Let \(\rho \in \mc{P}_1(\mathbb{R}^d)\) be an absolutely continuous measure and define 
\[
\Delta_N \coloneqq \{ Y=(y_1,\ldots,y_N)\in (\mathbb{R}^d)^N ~~;~~ \exists ~ 1\leq i<j\leq N \text{ such that } y_i = y_j \}.
\]
Given a vector of volumes \(v = (v_1, \ldots, v_N) \in \mathbb{R}_{>0}^N\) satisfying \(\sum_i v_i = 1\), for any \(Y \in (\mathbb{R}^d)^N \setminus \Delta_N\), 
there always exists a vector of weights \(\phi^* \in \mathbb{R}^N\) such that  
\begin{equation}\label{eq:volumes}
	\rho(L_i( \phi^*,Y)) = v_i, \quad \forall\, i = 1,\ldots,N. 
\end{equation}
Moreover, there is a unique Laguerre tessellation that satisfies this condition (see \cite[Theorem 3]{aurenhammer1998minkowski} or also \cite[Section 4]{MerigotThibertOT}). 

To see this, consider the following semi-discrete optimal transport problem:
\begin{equation}\label{eq:ot}
\max\left\{ \int_{\mathbb{R}^d \times \mathbb{R}^d} \langle x, y \rangle \, \mathrm{d}\gamma(x, y) \;;\; \gamma \in \Gamma(\rho, \nu_N(v,Y)) \right\},
\end{equation}
where \(\Gamma(\mu, \nu)\) denotes the set of probability measures on \(\mathbb{R}^d \times \mathbb{R}^d\) with marginals $\mu, \nu \in \mathcal{P}_1(\mathbb{R}^d)$,
and $\nu_N(v,Y)$ is the discrete measure supported on $Y$ with masses $v$, i.e.,
\begin{equation}\label{eq:nuYv}
\nu_N(v,Y) \coloneqq \sum_{i=1}^N v_i \delta_{y_i} \in \mathcal{P}(\mathbb{R}^d).
\end{equation}
The term semi-discrete refers here to the fact that one two marginals of the plan $\gamma$ is absolutely continuous while the other is discrete. Since \(\rho\) is absolutely continuous, Brenier's
Theorem \cite[Section 1.3.1]{SantambrogioBook}
guarantees that problem \eqref{eq:ot} admits a unique solution \(\gamma\), which is characterized as the only transport plan \(\gamma \in \mc{P}(\mathbb{R}^d \times \mathbb{R}^d)\) satisfying \(\gamma = (\mathrm{Id}, \nabla u)_\# \rho\) and 
\begin{equation}\label{eq:pushu}
(\nabla u)_\# \rho = \nu_N(v,Y),
\end{equation}
where \(u : \mathbb{R}^d \to \mathbb{R}\) is a convex function known as Brenier potential, and \(\#\) denotes the push-forward of measures. It can be shown that \(u\) necessarily takes the form
\begin{equation}\label{eq:optu}
u(x) = \max_{i\in \{1,\ldots,N\}} \left\{ \langle x, y_i \rangle - \phi_i \right\},
\end{equation}
for an appropriate vector of weights \(\phi\); see, e.g., \cite[Proposition 37]{MerigotThibertOT}. Consequently,
\[
\nabla u(x) = y_i \quad \text{for a.e. } x \in L_i(\phi,Y).
\]
By expressing condition \eqref{eq:pushu} with this ansatz, one deduces the existence of a vector of weights \(\phi^*\) satisfying \eqref{eq:volumes}. Choosing \(\phi = \phi^*\) in \eqref{eq:optu} yields a Brenier potential for the transport problem \eqref{eq:ot}. Moreover, the  plan $\gamma^*$ defined by
\begin{equation}\label{eq:gammastar}
\int_{\mathbb{R}^d\times \mathbb{R}^d} \varphi(x,y) \, \ed\gamma^*(x,y) = \sum_{i=1}^N \int_{L_i(\phi^*,Y)} \varphi(x,y_i) \, \ed \rho(x)\,,
\end{equation}
for all $\varphi \in C_b(\mathbb{R}^d \times \mathbb{R}^d)$, 
is the unique solution of problem \eqref{eq:ot}.

\subsection{Convex order} \label{sec:convexoreder}
Two probability measures $\nu,\rho \in \mc{P}_1(\mathbb{R}^d)$ are said to be in convex order, denoted $\nu \preceq_C \rho$, if and only if
\begin{equation}\label{eq:convexorder}
\int_{\mathbb{R}^d} \varphi \, \ed \nu \leq \int_{\mathbb{R}^d} \varphi \,\ed \rho
\end{equation}
for all convex functions $\varphi:\mathbb{R}^d\rightarrow \mathbb{R}$. By Strassen's Theorem \cite{strassen1965existence}, a convex order
relationship between two measures is equivalent to the existence of a martingale coupling between them (see Lemma \ref{lem:strassen}). This fact has made convex order relations a central tool in economics and finance applications, for example. In particular, common problems in this context are the design of sampling algorithms preserving convex order relations or, conversely, statistical tests to infer convex order relationships from samples. Recently, these problems have been approached using Wasserstein projections: sampling strategies were proposed in \cite{alfonsi2019sampling, alfonsi2020sampling}, while statistical tests for convex order were developed in \cite{kim2024statistical, kim2024backward}.

The Wasserstein projection of a given probability measure $\nu$ onto the set of measures dominated by $\rho$ in the convex order
is found by solving
\begin{equation}\label{eq:projection}
\inf \{ W^2_2(\mu,\nu)\,; \mu \preceq_C \rho \}
\end{equation}
where $W_2$ denotes the 2-Wasserstein distance (see equation \eqref{eq:w2} for the definition). Remarkably, this problem can also be recast as an instance of weak optimal transport \cite{gozlan2017kantorovich,gozlan2020mixture}; this is a class of optimization problems over couplings $\gamma \in \Gamma(\rho,\nu)$ as in \eqref{eq:ot}, but where the functional minimized is 
not linear in $\gamma$. 
{To be precise, problem \eqref{eq:projection} is equivalent to a barycentric weak optimal transport problem; see for example \cite[Proposition 1.1]{gozlan2020mixture} or \cite[Theorem 2.1]{alfonsi2019sampling}.}
This interpretation has been crucial for the recent development of numerical algorithms to solve projection problems as in \eqref{eq:projection} in the fully discrete setting, where all measures are discrete \cite{alfonsi2020sampling,gallouet2025metric,kim2024statistical}. 
To the best of our knowledge, no numerical strategies for the semi-discrete setting (where $\rho$ in equation \eqref{eq:projection} is absolutely continuous and $\nu$ is discrete) have been proposed prior to this work.

\subsection{Contributions and structure of the paper} Let $\mc{C}_N(v,\rho)$ be the set of particle positions $B =(b_1,\ldots,b_N)\in (\mathbb{R}^d)^N$ associated with discrete measures $\nu_N(v,B)$ that satisfy \eqref{eq:convexorder}, 
i.e.,
\[
\mc{C}_N(v,\rho) \coloneqq \{ B =(b_1,\ldots,b_N)\in (\mathbb{R}^d)^N \,;\, \nu_N(v,B) \preceq_C \rho\},
\]
where $\rho \in\mc{P}_1(\mathbb{R}^d)$ is a given absolutely continuous probability measure with bounded density with respect to Lebesgue, and as before \(v = (v_1, \ldots, v_N) \in \mathbb{R}_{>0}^N\) with \(\sum_i v_i = 1\). 

Using \eqref{eq:convexorder} and the definition of convexity one can verify that $\mc{C}_N(v,\rho)$ is a convex set. 
Our main contributions concern the structure of this set and its relation to the reconstruction problem stated in Question \ref{ques:reconstruction}. 
Specifically, we will show that the exposed points of the set $\mc{C}_N(v,\rho)$ are the vectors $B =(b_i)_i$ with
\[
b_i = \frac{1}{v_i} \int_{L_i(\phi^*,Y)} x \, \ed \rho(x) \,,
\]
where $Y = (\mathbb{R}^d)^N\setminus \Delta_N$ is any vector of generators and $\phi^*$ is such that condition \eqref{eq:volumes} holds. In other words, there exists  
$Y \in (\mathbb{R}^d)^N \setminus \Delta_N$ and $\phi \in \mathbb{R}^N$ such that, for all $i \in \{1,\ldots,N\}$, $L_i(\phi,Y)$ has measure $v_i$ and barycenter $b_i$ (with respect to $\rho$) if and only if $(b_1,\ldots,b_N)$ is an exposed point of $\mc{C}_N(v,\rho)$. This answers an open problem posed in \cite{bourne2024inverting}.
Moreover we show that the extreme points of $\mc{C}_N(v,\rho)$ can be constructed in a similar way but replacing Laguerre tessellations by an inductive type of tessellation introduced in \cite{akopyan2012kadets}, which we name hierarchical Laguerre tessellations. 

Similar results are already present in the literature in different contexts. In \cite{ekeland2014optimal}, 
Ekeland and Schachermayer already identified a connection between convex order and optimal transport that relies on interpreting random variables as exposing directions, which is essentially equivalent to our point of view. However, they do not consider specifically the semi-discrete case, which allows for a very explicit characterization of exposed and extreme points.
On the other hand, in the fully-discrete case, 
i.e., when also $\rho$ is a fixed discrete measure, the set $\mc{C}_N(v,\rho)$ coincides with the so-called \emph{gravity body} introduced by  Brieden and Gritzmann \cite{brieden2012optimal} to study balanced clusterings of points in $\mathbb{R}^d$ \cite{brieden2017constrained}. In that case, $\mc{C}_N(v,\rho)$ is a polytope and its vertices correspond to the barycenters of a certain class of Laguerre tessellations. 
Our results represent the natural extension of such a characterization to the case of an absolutely continuous reference measure $\rho$.

Exploiting our characterization of the set $\mc{C}_N(v,\rho)$, we  provide sufficient conditions that guarantee that, given a collection of barycenters $B$ of a Laguerre tessellation, one can recover a Laguerre tessellation whose cell barycenters are arbitrary close to $B$ by solving 
the
Wasserstein projection problem \eqref{eq:projection}.
Finally, we provide two numerical algorithms to
solve this reconstruction problem, or more generally problem \eqref{eq:projection} in the semi-discrete setting, 
i.e.,
whenever $\rho$ is absolutely continuous and $\nu$ is discrete. Importantly, our algorithm can also be applied to solve the \emph{fitting problem}, which corresponds to the case where  $B$ is not necessarily a collection of barycenters of a Laguerre tessellation.

The rest of the paper is structured as follows. In Section \ref{sec:hierarchical} we introduce the concept of {a} hierarchical Laguerre tessellation and describe its relation with classical Laguerre tessellations. In Section \ref{sec:convexorder} we characterize the exposed and extreme points of the set $\mc{C}_N(v,\rho)$. We also prove in Proposition~\ref{prop:uniqueness} that a hierarchical Laguerre tessellation is uniquely determined by the volumes and barycenters of its cells. This generalizes Theorem~4.5 in \cite{bourne2024inverting}. In Section \ref{sec:projection} we study the projection problem onto to $\mc{C}_N(v,\rho)$, and in Section \ref{sec:laguerre} we describe how one can use it to recover a Laguerre tessellation with prescribed cell volumes from the cell barycenters. Finally, we propose different convex optimization methods to solve the projection problem and show numerical results in Section \ref{sec:numerical}, including an application in materials science of fitting a Laguerre tessellation to an EBSD image of steel provided by Tata Steel Netherlands.

\section{Hierarchical Laguerre tessellations}\label{sec:hierarchical}

In this section we describe the concept of a hierarchical Laguerre tessellation, which has already appeared in the literature for different types of applications \cite{akopyan2012kadets,JaumeRote2016RecursivelyRegular}. Here, we focus on the case where the cell volumes of the tessellations are assigned, and describe the precise relation between hierarchical Laguerre tessellations and classical ones. This will be useful in the next section to characterize the set of discrete measures dominated in  convex order by a given measure.

We start by recalling that a convex tessellation of $\mathbb{R}^d$ is a list $(P_1, \ldots, P_N)$ of (possibly empty) closed convex sets $P_i$, called cells, with pairwise disjoint interiors such that
\[\bigcup_i {P_i} = \mathbb{R}^d.\] 
Let $\rho \in \mathcal{P}_1(\mathbb{R}^d)$ be a reference probability measure
{that is absolutely continuous with respect to the Lebesgue measure.}
We will also denote by $\rho$ its density with respect to the Lebesgue measure, which we assume to be bounded. 
We will be interested in  tessellations with fixed cell volumes with respect to $\rho$.
More precisely, given a vector ${v} =(v_1,\ldots,v_N)\in \mathbb{R}_{>0}^N$ with $\sum_i v_i=1$, we define the set of convex tessellations with volumes ${v}$ as follows:
\[
\mc{T}^N_{\mr{conv}}(v,\rho) \coloneqq \{ (P_i)_{i=1}^N \text{ convex tessellation of } \mathbb{R}^d \,;\, \rho(P_i) = v_i~ \forall i\}\,.
\]
Similarly, we define the set of Laguerre tessellations with volumes $v$ as follows:
\[
\mc{T}^N_{\mr{Lag}}(v,\rho) \coloneqq \{ (P_i)_{i=1}^N \text{ Laguerre tessellation of } \mathbb{R}^d \,;\, \rho(P_i) = v_i~ \forall i\}\,.
\]

Clearly $\mc{T}^N_{\mr{Lag}}(v,\rho)\subset \mc{T}^N_{\mr{conv}}(v,\rho)$. However there are convex  tessellations that are not Laguerre tessellations. Typical examples in any dimension are binary partitions by hyperplanes (see Figure \ref{fig:squares}, right). In $\mathbb{R}^2$, however, we have further counterexamples due to stronger geometric constraints implied by the definition of Laguerre tessellations, illustrated in Figure \ref{fig:cellcomplex}.

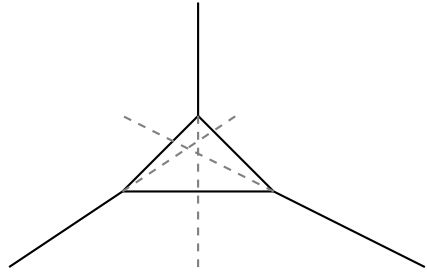
\begin{figure}
\begin{tikzpicture}
\draw[thick] (2,1) -- (4,1);
\draw[thick] (4,1) -- (3,2);
\draw[thick] (3,2) -- (2,1);
\draw[thick] (3,3.5) -- (3,2);
\draw[thick] (.5,0) -- (2,1);
\draw[thick] (6,0) -- (4,1);
\draw[thick,dashed,gray] (2,1) -- (3.5,2);
\draw[thick,dashed,gray] (4,1) -- (2,2);
\draw[thick,dashed,gray] (3,2) -- (3,0);
\end{tikzpicture}
\vspace{-2em}
\caption{An example of convex tessellation of $\mathbb{R}^2$ which is not a (hierarchical) Laguerre tessellation. The two-dimensional cells of the tessellation have as boundaries the solid lines. In order for this tessellation to be Laguerre, it should be possible to recover it as the projection of a polyhedron in 
$\mathbb{R}^3$, but this is impossible if the dashed lines do not meet at a single point. In fact if three planes intersect in $\mathbb{R}^3$, their common intersection is either a line or a single point.} \label{fig:cellcomplex}
\end{figure}

By identifying the convex cells of a given tessellation with the associated characteristic
functions, we can interpret $\mc{T}^N_{\mr{conv}}$ and $\mc{T}^N_{\mr{Lag}}$ as subsets of $(L^\infty(\rho))^N$, which we equip with the product weak-* topology. 
Using 
this
topology, we now characterize the closure of $\mc{T}^N_{\mr{Lag}}(v,\rho)$ in terms of hierarchical Laguerre tessellations. 

\begin{definition}[Hierarchical Laguerre Tessellation]
A hierarchical Laguerre tessellation is defined recursively as follows:
\begin{itemize}
    \item Every Laguerre tessellation is a hierarchical Laguerre tessellation;
 \item If $(P_i)_{1\leq i\leq N}$ is a hierarchical Laguerre tessellation with $N$ cells and if $(Q_j)_{1\leq i\leq M}$ is a Laguerre tessellation with $M$ cells, then for any $i_0 \in\{1,\hdots,N\}$, the tessellation with  $N+M-1$ cells 
 $P_1,\hdots, P_{i_0-1}, P_{i_0} \cap Q_1, P_{i_0} \cap Q_2, \hdots, P_{i_0} \cap Q_{M}, P_{i_0+1}, \hdots, P_{N}$
is also a hierarchical Laguerre tessellation.
\end{itemize}
\end{definition}
Informally, given a hierarchical Laguerre tessellation $(P_i)_i$, the tessellation obtained by partitioning a given cell $P_i$ with a Laguerre tessellation is also a hierarchical Laguerre tessellation; see Figure \ref{fig:squares} (right).
In particular, any recursive binary partition of $\mathbb{R}^d$ by hyperplanes defines a hierarchical Laguerre tesselation.
As for Laguerre tessellations, we denote the set of hierarchical Laguerre tessellations with $N$ cells of  prescribed volumes as follows:
\[
\mc{T}^N_{\mr{hLag}}(v,\rho) \coloneqq \{ (P_i)_{i=1}^N \text{ hierarchical Laguerre tessellation of } \mathbb{R}^d \,;\, \rho(P_i) = v_i~ \forall i\}\,.
\]

\begin{proposition}\label{prop:closure} $\overline{\mc{T}^N_{\mr{Lag}}(v,\rho)}=\mc{T}^N_{\mr{hLag}}(v,\rho)$.
\end{proposition}

\begin{proof}

Let us show that $\overline{\mc{T}^N_{\mr{Lag}}(v,\rho)}\subseteq\mc{T}^N_{\mr{hLag}}(v,\rho)$.
Given any sequence of tessellations $((L_i^n)_i)_n\subset \mc{T}^N_{\mr{Lag}}$, let 
{$(Y^n)_n \subset (\mathbb{R}^d)^N \setminus \Delta_N$}
be the associated sequence of generators, shifted and normalized so that $y^n_1=0$ and $\max_{i\neq k} |y_i^n-y_k^n| =1$.
(This is always possible thanks to the invariance properties of Laguerre tessellations with respect to translations and dilations of the generators; see also \cite[Proposition 6]{Meyron2019}.)
Up to extracting a subsequence (not relabelled), for some $2\leq M \leq N$ and $Y^\infty \in (\mathbb{R}^d)^M \setminus \Delta_M$, and a surjective map $\sigma:\{1,\ldots,N\} \rightarrow \{1,\ldots,M\}$, 
\[
y^n_i \rightarrow y^\infty_{\sigma(i)} \,, \quad \forall i=1,\ldots, N.
\]
Moreover,
\begin{equation}
\label{eq: Proof Prop 2.2 (2)}
\nu_N(v,Y^n) = \sum_{i=1}^N  v_i \delta_{y_i^n} \rightharpoonup \nu_M(w,Y^\infty)= \sum_{j=1}^M  w_j \delta_{y_i^\infty} \,,\quad w_j = \sum_{i\in\sigma^{-1}(j)} v_i,
\end{equation}
{where the convergence is understood as weak convergence of measures, i.e., in duality with continuous bounded functions on $\mathbb{R}^d$.}
Let $(\gamma^n)_n$ be the sequence of optimal transport plans from $\rho$ to $\nu_N(v,Y^n)$, obtained by solving problem \eqref{eq:ot}. By stability of optimal plans (see, e.g., Theorem 5.20 in \cite{villani2009optimal}), this converges weakly to the unique optimal transport plan from $\rho$ to $\nu_M(w,Y^\infty)$,
\[
\gamma = \sum_{j=1}^M 
{\rho|_{L^\infty_j} \otimes  \delta_{y_j^\infty}}  \,,
\]
where $(L^\infty_j)_j$ is the Laguerre tessellation with generators $Y^\infty$ and volumes ${w}$. In particular, denoting
\begin{equation}\label{eq:qjn}
Q^n_j \coloneqq \bigcup_{i \in \sigma^{-1}(j)} L_i^n\,.
\end{equation}
then $\rho|_{Q^n_j} \rightharpoonup \rho|_{L^\infty_j}$ weakly as $n \rightarrow \infty$, from which we can deduce that also the corresponding characteristic
functions converge, i.e.,
\[
 \mathbf{1}_{Q^n_j} \rightarrow \mathbf{1}_{L^\infty_j} \quad \text{weakly-* in } L^\infty(\rho).
\]
If $M=N$, we would be already done, since $(L_j^\infty)_j$ is a Laguerre tessellation. Otherwise, let $1\leq j \leq M$ be any index such that $\tilde{N}\coloneqq |\sigma^{-1}(j)| >1$ and define $Q^n_j$ as above.
We shift and normalize the set $\tilde{Y}^n = (y_i^n)_{i\in\sigma^{-1}(j)}$ so that  $y^n_{i_1} =0$ for some $i_1\in \sigma^{-1}(j)$ and 
$\max \{ |y_i^n-y_k^n| \,;\, i,k \in \sigma^{-1}(j), \; i \ne k \} = 1$.
Up to extracting a subsequence (not relabelled), for some $2\leq \tilde{M} \leq \tilde{N}$ and $\tilde{Y}^\infty \in (\mathbb{R}^d)^{\tilde{M}} \setminus \Delta_{\tilde{M}}$, and a surjective map $\sigma:\{1,\ldots,\tilde{N}\} \rightarrow \{1,\ldots,\tilde{M}\}$, 
\[
y^n_i \rightarrow y^\infty_{\sigma(i)} \,, \quad \forall i=1,\ldots, \tilde{N}.
\]
We can then proceed as before, with $\rho$ replaced by $\rho|_{Q^n_j}$, for which $\rho|_{Q^n_j} \rightharpoonup \rho|_{L^\infty_j}$ as $n\rightarrow \infty$. 
Iterating this argument implies
that $\overline{\mc{T}^N_{\mr{Lag}}}\subseteq \mc{T}^N_{\mr{hLag}}$.

Let us prove that ${\mc{T}^N_{\mr{hLag}}}\subseteq \overline{\mc{T}^N_{\mr{Lag}}} $. Given any hierarchical Laguerre tessellation $(L_i)_i\in \mc{T}^N_{\mr{hLag}}(v,\rho)$, using Lemma 3.5 in \cite{akopyan2012kadets}, we can construct a sequence of Laguerre tessellations $((\tilde{L}_i^n)_i)_n$ 
with generators $Y^n \in (\mathbb{R}^d)^N \setminus \Delta_N$ (but possibly different volumes) such that $\tilde{L}_i^n \rightarrow L_i$  locally in the Hausdorff metric. Since $\tilde{L}_i^n$ and $L_i$ are convex sets, this implies 
\begin{equation}\label{eq:linconv}
 \mathbf{1}_{\tilde{L}^n_i} \rightarrow \mathbf{1}_{L_i} \quad \text{weakly-* in } L^\infty(\rho);
\end{equation}
see \cite{beer1974hausdorff}, for example.
In particular, this means that \[v_i^n \coloneqq \rho(\tilde{L}_i^n)\rightarrow v_i\] for all $i$.
Let $({L}_i^n)_i$ be the Laguerre tessellation with generators $Y^n$, but with weights chosen so that $\rho({L}_i^n)=v_i$ for all $i$. We now show that, up to the extraction of a subsequence, ${L}^n_i$ also converges towards $L_i$ for all $i$, which implies that $\mc{T}^N_{\mr{hLag}}\subseteq \overline{\mc{T}^N_{\mr{Lag}}}$.
Proceeding as before, we can assume that, for some $2\leq M \leq N$ and $Y^\infty \in (\mathbb{R}^d)^M \setminus \Delta_M$, and a surjective map $\sigma:\{1,\ldots,N\} \rightarrow \{1,\ldots,M\}$, 
\[
y^n_i \rightarrow y^\infty_{\sigma(i)} \,, \quad \forall i=1,\ldots, N.
\]
Denoting ${v}^n \coloneqq (v_i^n)_i$, and $w$ as defined in equation \eqref{eq: Proof Prop 2.2 (2)}, we have
\[
\nu_N(v^n,Y^n) \rightharpoonup \nu_M(w,Y^\infty) \,, 
\]
and also
\[
\nu_N(v,Y^n) \rightharpoonup \nu_M(w,Y^\infty).
\]
Using again the stability of optimal transport maps we deduce that 
\[
 \mathbf{1}_{Q^n_j} \rightarrow \mathbf{1}_{L^\infty_j} \quad \text{weakly-* in } L^\infty(\rho),
\]
where $Q^n_j$ is again defined as in \eqref{eq:qjn} 
and $(L_j^\infty)_j$ is  the Laguerre tessellation with volumes ${w}$ and generators $Y^\infty$. Moreover, due to \eqref{eq:linconv}, we also have
\[
L_j^\infty = \bigcup_{i\in \sigma^{-1}(j)} L_i\,.
\]
Iterating the argument as before we obtain the result.
\end{proof}

\section{Semi-discrete convex order}
\label{sec:convexorder}
In this section we study the set of discrete measures of prescribed masses in convex order with respect to $\rho$. We will give a characterization of this set as a convex body 
in
$\mathbb{R}^{dN}$, and in terms of measures supported on the cell barycenters of Laguerre and hierarchical Laguerre tessellations.

As before, we fix a vector $v =(v_1,\ldots,v_N)\in  \mathbb{R}_{>0}^N$ such that $\sum_i v_i = 1$, and 
an absolutely continuous
reference measure $\rho\in\mc{P}_1(\mathbb{R}^d)$ with bounded density.
We study the following set:  
\[
\mc{C}_N(v,\rho) \coloneqq \{ B\in (\mathbb{R}^{d})^N ~;~  \nu_N(v,B) \preceq_C \rho\} \subset \mathbb{R}^{dN},
\]
where the notation $\preceq_C$ was defined in equation \eqref{eq:convexorder} and
we recall that
\[
\nu_N(v,B) \coloneqq \sum_{i=1}^N v_i \delta_{b_i} \in \mathcal{P}(\mathbb{R}^d)\,,
\]
for any $B=(b_1,\ldots,b_N) \in (\mathbb{R}^{d})^N$.
For brevity, we will omit to indicate the dependency of $\mc{C}_N$ on $\rho$ and ${v}$ when it is clear from the context. It will be  convenient to use the weighted $l^2$-metric on $(\mathbb{R}^d)^N$ defined by
\[
\langle Y,Z\rangle_v \coloneqq \sum_i \langle y_i,z_i\rangle v_i,
\]
for any $Y=(y_i)_i$ and $Z=(z_i)_i$ in $(\mathbb{R}^d)^N$. Moreover, we will denote $\|Y\|_v^2 \coloneqq \langle Y,Y\rangle_v$.

\subsection{Basic properties of $\mc{C}_N$}

By definition, a point $B\in \mc{C}_N$ if and only if
\[
\sum_i \varphi(b_i) v_i \leq 
{\int_{\mathbb{R}^d} \varphi \, \ed \rho}
\]
for any convex function $\varphi:\mathbb{R}^d\rightarrow \mathbb{R}$. 
In particular, $\mc{C}_N$ is convex. 
It is also non-empty because, by Jensen's inequality, $(\mr{bary}(\rho),\ldots,\mr{bary}(\rho)) \in \mathcal{C}_N(v,\rho)$, where 
\[
\mr{bary}(\rho) \coloneqq \int_{\mathbb{R}^d} x \, \ed \rho(x)\,.
\]
The set $\mathcal{C}_N$
is also closed, since the functions $\varphi$ defining the set are continuous.  Moreover $\mc{C}_N$ is bounded, since we can 
take $\varphi(x)=|x|$.
Now, if $B \in \mc{C}_N$, 
by taking $\varphi(x) = \pm x$,
we obtain that 
\begin{equation}
\label{eq: 1 is normal to C_N}
\sum_i b_i v_i = \mr{bary}(\rho) \,.
\end{equation}
Due to this constraint $\mc{C}_N$ cannot have full dimension, and {more precisely $\mathrm{dim}(\mc{C}_N) \leq (N-1)d$. 
{(Recall that the dimension of a convex set is defined to be the dimension of its affine hull.)}
Since we are supposing that  $\rho$ is absolutely continuous, the following equality holds:}
\begin{lemma}\label{lem:dimension}
$\mathrm{dim}(\mc{C}_N)=(N-1)d$.
\end{lemma}

\begin{proof} 
We argue by induction on $N$. 
The statement is trivially true for $N=1$. 
 Suppose it is also true for a given $N$, and fix a vector of admissible volumes $v=(v_1,\ldots,v_{N+1})$.  
Let $\tilde{\rho}\in L^\infty(\mathbb{R}^d)$ be any density such that $0\leq \tilde{\rho}\leq \rho$ and $\tilde{\rho}(\mathbb{R}^d)=v_{N+1}$.
{Let $\tilde{v}=( v_1 , \ldots, v_N )$.}
By Jensen's inequality, we can construct $B =(b_1,\ldots,b_{N+1}) \in \mc{C}_{N+1}(v,\rho)$ by setting 
\[b_{N+1} = \mr{bary}(\tilde{\rho})
\coloneqq \frac{1}{\tilde{\rho}(\mathbb{R}^d)} \int_{\mathbb{R}^d} x \, \ed \tilde{\rho}(x),
\]and {taking any} $\tilde{B}=(b_1,\ldots,b_N) {\in \mathcal{C}_N(\tilde{v},\rho - \tilde{\rho})}$ so that
\[
\nu_N(\tilde{{v}},\tilde{B}) \preceq_C \rho -\tilde{\rho}\,.
\]
Hence we just need to show that $\tilde{\rho}$ can be chosen so that 
$b_{N+1} = \mathrm{bary}(\tilde{\rho})$ is an arbitrary point in a set of dimension $d$, i.e., 
$\{  \mathrm{bary}(\tilde{\rho}) \,;\, \tilde{\rho} \in L^\infty(\mathbb{R}^d), \;  0\leq \tilde{\rho}\leq \rho, \; \tilde{\rho}(\mathbb{R}^d)=v_{N+1} \}$
has dimension $d$, as
this implies that $\mathrm{dim}(\mc{C}_{N+1})\geq dN$.
By convexity, it suffices to find $d+1$ densities $\{\tilde{\rho}_i\}_{i=1}^{d+1}$ as above
{(i.e., $\tilde{\rho}_i \in L^\infty(\mathbb{R}^d)$,  $0\leq \tilde{\rho}_i \leq \rho$,  $\tilde{\rho}_i(\mathbb{R}^d)=v_{N+1}$)}
such that the convex hull of $b_1,\ldots,b_{d+1}$, with $b_i = \mr{bary}(\tilde{\rho}_i)$, has dimension $d$. We can construct them as follows.  

First, we define the sets $A_i\subset \mathbb{R}^d$ inductively by
\[
A_1 \coloneqq \mathbb{R}^d\,, \quad p_i \coloneqq \mr{bary}(\rho|_{A_i})\,, \quad A_{i+1} \coloneqq A_{i}\cap \{x {\in \mathbb{R}^d} \,:\, \langle h_{i},x-p_i\rangle >0\}\,,
\]
for $i=1,\ldots,d$,
where $h_1 = \mathbb{R}^d\setminus\{0\}$ is arbitrary and, for all $i \geq 2$, $h_{i} \in \mr{span}(p_2-p_1,\ldots, p_{i}-p_1)^\perp\setminus \{0\}$  is also chosen arbitrarily. It is easy to check that $\rho(A_i)> 0$: in fact, if $\rho(A_i)>0$, then $\rho(A_{i+1})>0$ independently of the choice of $h_i$, since $p_i =\mathrm{bary}(\rho|_{A_i})$. (If this were not the case, then necessarily $\rho( \partial A_{i+1})>0$, but this is impossible since $\rho$ is absolutely continuous).
Then we define
\[
\tilde{\rho}_i \coloneqq \frac{v_{N+1}}{1-\varepsilon_i}\left( \rho - \varepsilon_i \frac{\rho|_{A_i}}{\rho(A_i)} \right)
\]
where $0<\varepsilon_i<\min( \rho(A_i),1-v_{N+1})$ is a given coefficient. By the bounds on $\varepsilon_i$, we have indeed $0\leq \tilde{\rho}_i\leq \rho$ and moreover $\tilde{\rho}_i(\mathbb{R}^d) = v_{N+1}$. 
Now, since $A_1 = \mathbb{R}^d$, then $b_1=p_1$ and, for all $i \ge 2$, 
\begin{equation}
\label{eq: b_i - b_1}
{b_i - b_1 = \mr{bary}(\tilde{\rho}_i) - b_1 = \frac{1}{1-\varepsilon_i} (p_1 - \varepsilon_i p_i) - p_1
= \frac{\varepsilon_i}{\varepsilon_i-1} (p_i - p_1).
}
\end{equation}
Since $\rho$ is absolutely continuous and $A_{i+1}$ is convex, then $p_{i+1}=\mr{bary}(\rho|_{A_{i+1}})\in A_{i+1}$. By definition of $A_{i+1}$, this means that for all $i\geq 2$, 
\[
 p_{i+1}-p_i\notin \mr{span}(p_2-p_1,\ldots, p_{i}-p_1)\,,
 \]
which also implies that $p_{i+1}-p_1 \notin\mr{span}(p_2-p_1,\ldots, p_{i}-p_1)$.
Hence, by induction and equation \eqref{eq: b_i - b_1}, the vectors $
\{b_2-b_1,\ldots, b_{d+1}-b_1\}$
are linearly independent, and we are done.
\end{proof}

\subsection{Strassen's Theorem and barycenters of Laguerre tessellations} 
{
As  can already be deduced from the proof of Lemma \ref{lem:dimension}, one can construct points in $\mc{C}_N$ by decomposing the density $\rho$ appropriately. In order to describe this procedure, let us first recall the following alternative characterization of convex order, due to Strassen \cite{strassen1965existence}:
\begin{lemma}[Strassen's Theorem] \label{lem:strassen} Let $\mu, \nu \in \mc{P}_1(\mathbb{R}^d)$. Then $\mu \preceq_C \nu$ if and only if there exists a martingale coupling between $\mu$ and $\nu$, namely, there exists
a coupling $\theta \in \Gamma(\mu,\nu)$ such that $\ed \theta(x,y) = \ed \theta_x(y) \ed \mu(x)$ (in the sense of disintegration of measures) and
\begin{equation}\label{eq:martingale}
\int_{\mathbb{R}^d} y \, \ed \theta_x(y) = x \quad \text{ for $\mu$-a.e. } x\in\mathbb{R}^d\,.
\end{equation}
\end{lemma}
In our setting 
($\mu = \nu_N(v,B)$, $\nu = \rho$),
this means that 
$B \in \mc{C}_N(\rho,v)$ if and only if there exists $N$ densities $(\rho_i)_{i=1}^N$ such that 
\begin{equation}\label{eq:decomposition}
\sum_{i=1}^N \rho_i = \rho \,, \quad \int_{\mathbb{R}^d} \, \ed  \rho_i = v_i 
\,, \quad b_i = \mr{bary}(\rho_i)\coloneqq\frac{1}{v_i} \int_{\mathbb{R}^d} y \, \ed \rho_i({y})\,,
\end{equation}
for all $i=1,\ldots, N$, or in other words, the points $b_i$ are the barycenters of the densities $\rho_i$. 
{In the notation of Lemma \ref{lem:strassen}, $\theta_{b_i} = \rho_i/v_i$ for all $i$.}
 Note, in particular, that if $v$ and $B$  
 {satisfy}  \eqref{eq:decomposition}, then clearly $\nu_N(v,B) \preceq_C \rho$ by Jensen's inequality. An illustration of this construction is shown in Figure \ref{fig:decomposition}.
}

\begin{figure}
\centering
    \begin{tikzpicture}
      \begin{axis}[
        scale only axis,              % Use only the axis area for sizing
        width=12cm, height=4cm,
        xmin=-3.5, xmax=2.5,           % Center the plot so that b₂ (-0.5) is near the center
        ymin=0, ymax=2.5,              % Set the vertical range explicitly
        enlarge y limits=false,        % Prevent extra vertical space
    %ymin=0, ymax=3.5,              % Manually set the y-range to scale it down
    axis x line=middle, axis y line=none,
    axis on top, % draw the axis lines above the filled areas
    xlabel={$x$},
    domain=-3:3, samples=200,
    xtick=\empty, ytick=\empty,
    enlargelimits=true,
    clip=false,
  ]
    % Baseline (y = 0)
    \addplot[name path=F0, draw=none] coordinates {(-3,0) (3,0)};
    
    % Cumulative curves:
    % F1 = ρ₁
    \addplot[name path=F1, draw=none, domain=-3:3] {0.6*exp(-((x+2)^2)/0.3)};
    % F2 = ρ₁ + ρ₂
    \addplot[name path=F2, draw=none, domain=-3:3] {0.6*exp(-((x+2)^2)/0.3) + 1.2*exp(-((x+0.5)^2)/0.8)};
    % F3 = ρ₁ + ρ₂ + ρ₃ = ρ
    \addplot[name path=F3, draw=none, domain=-3:3] {0.6*exp(-((x+2)^2)/0.3) + 1.2*exp(-((x+0.5)^2)/0.8) + 0.8*exp(-((x-1)^2)/0.5)};
    
    % Fill between curves:
    % Fill from baseline to F1 with pastel blue (ρ₁)
    \addplot[fill=blue!10] fill between[of=F0 and F1];
    % Fill from F1 to F2 with pastel red (ρ₂)
    \addplot[fill=red!10] fill between[of=F1 and F2];
    % Fill from F2 to F3 with pastel green (ρ₃)
    \addplot[fill=green!10] fill between[of=F2 and F3];
    
    % Draw the total density curve ρ in black
    %\addplot[black, thick, domain=-3:3] {0.6*exp(-((x+2)^2)/0.3) + 1.2*exp(-((x+0.5)^2)/0.8) + 0.8*exp(-((x-1)^2)/0.5)};
    
    % Labels for each density contribution:
    \node[blue!70!black] at (axis cs:-2,0.35) {$\rho_1$};
    \node[red!90!black] at (axis cs:-0.5,.9) {$\rho_2$};
    \node[green!70!black] at (axis cs:1,0.6) {$\rho_3$};
    %\node[black] at (axis cs:0,1.5) {$\rho=\rho_1+\rho_2+\rho_3$};

    % Mark barycenters on the x-axis:
    \filldraw[blue!70!black] (axis cs:-2,0) circle (2pt) node[below] {$b_1$};
    \filldraw[red!90!black] (axis cs:-0.5,0) circle (2pt) node[below] {$b_2$};
    \filldraw[green!70!black] (axis cs:1,0) circle (2pt) node[below] {$b_3$};
  \end{axis}
\end{tikzpicture}
\vspace{-3em}
\caption{Construction of a discrete measure in convex order with a density $\rho = \rho_1+\rho_2 +\rho_3$ on $\mathbb{R}$. Setting $b_i = \mr{bary}(\rho_i)$ and $v_i = \rho_i(\mathbb{R})$, then $\nu_3(v,B) \preceq_C \rho$. }\label{fig:decomposition}
\end{figure}
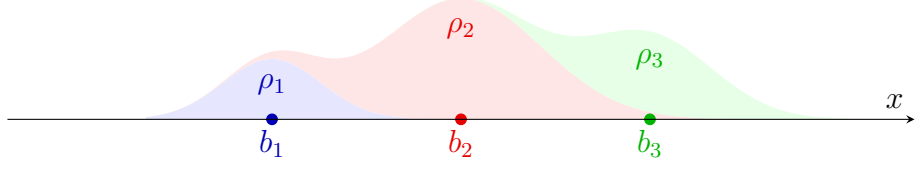

As a particular case of this construction, given 
any measurable partition $(P_i)_{i=1}^N$ of $\mathbb{R}^d$ such that $v_i = \rho(P_i)$ for all $i$,
we can also construct an element of $B\in \mc{C}_N$ simply by setting 
\[
b_i = \mathrm{bary}_{\rho}(P_i) \coloneqq \frac{1}{v_i} \int_{P_i} x \, \ed \rho(x).
\]
In particular, denoting by
\[
\mc{B}^N_{\mr{Lag}}(v,\rho) \coloneqq \{ B \in (\mathbb{R}^{d})^N ~;~ \exists\, (L_i)_i \in \mc{T}^N_{\mr{Lag}}(v,\rho) \text{ such that } b_i = \mathrm{bary}_{\rho}(L_i) ~ \forall \, i\}
\]
and 
\[
\mc{B}^N_{\mr{hLag}}(v,\rho) \coloneqq \{ B \in (\mathbb{R}^{d})^N ~;~ \exists\, (L_i)_i \in \mc{T}^N_{\mr{hLag}}(v,\rho) \text{ such that } b_i = \mathrm{bary}_{\rho}(L_i) ~ \forall \, i\}
\]
the set of barycenters of Laguerre and hierarchical Laguerre tesellations with prescribed volumes, we have $\mc{B}^N_\mr{Lag} \subset\mc{B}^N_\mr{hLag} \subset \mc{C}_N$ by construction.
Furthermore, $\mc{B}^N_{\mr{Lag}}$ is neither convex nor closed. An example showing that the limit of vectors in $\mc{B}^N_{\mr{Lag}}$ may not belong to the same set is shown in Figure \ref{fig:squares}. More precisely, as a corollary of Proposition \ref{prop:closure}
we have the following:

\begin{lemma}\label{lem:closurebary} $\overline{\mc{B}^N_\mr{Lag}} = \mc{B}^N_\mr{hLag}$.
\end{lemma}

{
The following proposition states that there  exists a one-to-one correspondence between $\mc{T}^N_{\mr{hLag}}$ and $\mc{B}^N_{\mr{hLag}}$, generalizing an analogous result for Laguerre tessellations proved in \cite[Theorem 4.5]{bourne2024inverting}.

\begin{proposition} \label{prop:uniqueness} Let $B \in \mc{B}^N_{\mr{hLag}}(v,\rho)$. Then there exists a unique collection of probability measures $(\theta_i)_{i=1}^N \subset \mc{P}_1(\mathbb{R}^d)$  such that
\begin{equation}\label{eq:conditions}
\sum_i v_i \theta_i  = \rho\,, \quad \mr{bary}(\theta_i) \coloneqq  \int_{\mathbb{R}^d}x\ed \theta_i(x) = b_i \,, \quad \forall\, i =1,\ldots, N.
\end{equation}
In particular, there exists a unique tessellation of $\mathbb{R}^d$ with volumes and barycenters equal to $v$ and $B$ respectively, which therefore must be a hierarchical Laguerre tessellation.
\end{proposition}
\begin{proof} Let $(\tilde{L}_i)_i$ be a given hierarchical Laguerre tessellation corresponding to $B$. The existence of a $(\theta_i)_i$ as in the statement follows from the fact that we can always take, for $i=1,\ldots,N$, $\theta_i = \frac{1}{v_i} \rho|_{\tilde{L}_i}$. We want to prove that $(\theta_i)_i$ is uniquely determined by the conditions \eqref{eq:conditions}, which implies that $(\tilde{L}_i)_{i=1}^N$ is also unique.

Since $B\in \mc{B}^N_{\mr{hLag}}(v,\rho)$, there exists a Laguerre tessellation $(L_j)_{j=1}^M$ with generators $Y \in (\mathbb{R}^d)^M \setminus \Delta_M$, with $2\leq M\leq N$, and a surjective map $\sigma:\{1,\ldots,N\}\rightarrow \{1,\ldots, M\}$ such that, for all $j=1,\ldots,M$,
\[
 \rho(L_j) = \sum_{i \in \sigma^{-1}(j)} v_i \eqqcolon w_j \,, \quad  \mr{bary}_\rho(L_j)= \frac{1}{w_j} \sum_{i \in \sigma^{-1}(j)} v_i b_i  \,.
\]
By construction, the coupling $\gamma^* \in \mc{P}_1(\mathbb{R}^d\times \mathbb{R}^d)$ defined by
\[
\int_{\mathbb{R}^d\times \mathbb{R}^d} \varphi(x,y) \, \ed \gamma^*(x,y) = \sum_{j=1}^M \int_{L_j} \varphi(x,y_j) \, \ed \rho(x) \quad \forall\, \varphi \in C_b(\mathbb{R}^d\times \mathbb{R}^d)
\]
is the unique optimal transport plan from $\rho$ to $\nu_M(w,Y)$. 
Now, let $(\theta_i)_i \subset \mc{P}_1(\mathbb{R}^d)$ be an arbitrary collection of probability measures satisfying \eqref{eq:conditions}. We can define a coupling $\gamma \in \Gamma(\rho, \nu_M(w,Y))$ as follows:
\[
\int_{\mathbb{R}^d\times \mathbb{R}^d} \varphi(x,y) \, \ed \gamma(x,y) = \sum_{j=1}^M \sum_{i \in \sigma^{-1}(j)} v_i \int_{
\mathbb{R}^d} \varphi(x,y_j) \, \ed \theta_i(x)\quad \forall\, \varphi \in C_b(\mathbb{R}^d\times \mathbb{R}^d)\,.
\]
In particular, it is easy to check that
\[
\int_{\mathbb{R}^d\times \mathbb{R}^d} \langle x, y\rangle \, \ed \gamma(x,y) =\sum_{j=1}^M \sum_{i \in \sigma^{-1}(j)} v_j  \langle b_i,y_j\rangle   = \int_{\mathbb{R}^d\times \mathbb{R}^d} \langle x, y\rangle \, \ed \gamma^*(x,y)\,,
\]
and therefore by the uniqueness of $\gamma^*$, for all $j=1,\ldots, M$,
\[
\sum_{i \in \sigma^{-1}(j)} v_i \theta_i = \rho|_{L_j}\,.
\]
This means that the left-hand side is uniquely determined for all $j$.
But then we can iterate the construction replacing $\rho$ with $\rho|_{L_j}$, for each $j= 1, \ldots, M$. Eventually, this allows us to deduce that each $\theta_i$ is uniquely determined, and we are done.
\end{proof}}

\begin{figure}
\begin{tikzpicture}

% =======================
% Former RIGHT figure (now on the LEFT)
% =======================
\begin{scope}[shift={(-5,0)}]  % shift left

% Draw the second square
\fill[cyan!40, fill opacity=0.5] (4+1,0) rectangle (7+1,3);

\draw[thick] (4.9+1-.0533,-.4) -- (5.1+1,1.5) -- (4.9+1-.0533,3.4);
\draw[thick] (5.1+1,1.5) -- (7+1.7,1.5);

\fill[red] (4.5,1.5) circle (2pt);
\node[above,red] at (4.5,1.5) {$y_1$};

\fill[red] (8.5,2.03) circle (2pt);
\node[right,red] at (8.5,2.03) {$y_3$};

\fill[red] (8.5,0.97) circle (2pt);
\node[right,red] at (8.5,0.97) {$y_2$};

\draw[thick, dashed,red!50] (4.5,1.5) -- (8.5,2.03);
\draw[thick, dashed,red!50] (4.5,1.5) -- (8.5,0.97);
\draw[thick, dashed,red!50] (8.5,2.03) -- (8.5,0.97);

\end{scope}

% =======================
% Former LEFT figure (now on the RIGHT)
% =======================
\begin{scope}[shift={(6,0)}]  % shift right

\fill[cyan!40, fill opacity=0.5] (0,0) rectangle (3,3);
\draw[thick] (1,-.5) -- (1,3.5);
\draw[thick] (1,1.5) -- (3.7,1.5);

\fill[black] (.5,1.5) circle (2pt);
\node[above] at (.5,1.5) {$b_1$};

\fill[black] (2,.75) circle (2pt);
\node[right] at (2,.75) {$b_2$};

\fill[black] (2,2.25) circle (2pt);
\node[right] at (2,2.25) {$b_3$};

\end{scope}

\end{tikzpicture}
\vspace{-2em}
\caption{
Consider the Laguerre tessellation on the left, defined by the generators $(y_1,y_2,y_3)$ and volumes $v_1=v_2=v_3$, where $\rho$ is uniform on a square (shaded area). As $y_2$ and $y_3$ approach each other, the tessellation degenerates to the configuration on the right.
This limit is not a Laguerre tessellation, but a hierarchical Laguerre tessellation: it can be constructed by first partitioning the square into two cells via a Laguerre tessellation (corresponding to the vertical segment), and then subdividing the right  cell using a second Laguerre partition (corresponding to the horizontal segment). In particular, note that the vector of barycenters $B=(b_1,b_2,b_3) \in\mc{B}^N_\mr{hLag}$ is not in $\mc{B}^N_\mr{Lag}$. 
}\label{fig:squares}
\end{figure}
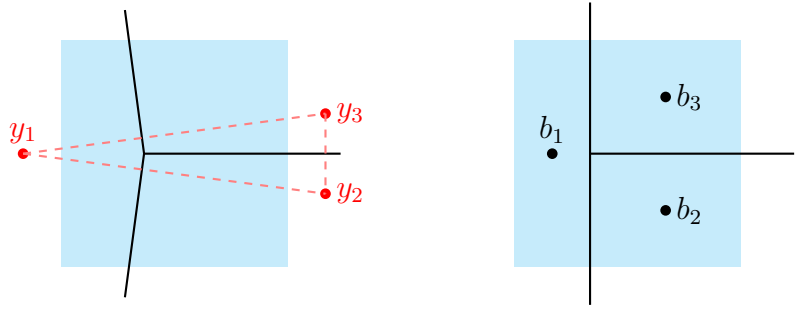

\subsection{Support function, exposed and extreme points} In this section we compute the support function of $\mc{C}_N$, and use this to prove that $\mc{B}^N_{\mr{Lag}}$ and $\mc{B}^N_{\mr{hLag}}$ are precisely the sets of exposed and extreme points
of $\mc{C}_N$, respectively.
 
Define the function $F_N: (\mathbb{R}^d)^N\rightarrow \mathbb{R}$ as follows:
\begin{equation}\label{eq:FY}
F_N(Y) \coloneqq \max\left\{ \int_{\mathbb{R}^d \times \mathbb{R}^d} \langle x,y\rangle \, \ed \gamma(x,y) ~~;~~\gamma \in \Gamma (\rho,\nu_N(v,Y)) \right\}.
\end{equation}
We will also denote the same function by $F_N(v,\rho;Y)$ when we want to
emphasize the dependence on ${v}$ and $\rho$.
Note that (if $\rho$ has finite second moments) $F_N$ is related to the Wasserstein distance $W_2(\rho,\nu_N(v,Y))$ via
\[
F_N(Y) = - \frac 12 W_2^2(\rho,\nu_N(v,Y)) + \frac 12 \int_{\mathbb{R}^d} |x|^2 \, \ed \rho(x) 
+ \frac 12 \sum_{i=1}^N v_i |y_i|^2.
\]
{The following proposition states that $F_N^*$ is the indicator function of the convex set $\mc{C}_N$. Therefore $F_N = F_N^{**}= \iota_{\mc{C}_N}^*$ is the support function of $\mc{C}_N$ by, e.g., \cite[Example 18.3]{BauschkeMoursi2023}.}

\begin{proposition}\label{prop:dualF} The function $F_N$ is convex and continuous; {it is non-differentiable on $\Delta_N$} and $C^1$ on $(\mathbb{R}^d)^N\setminus \Delta_N$, and for all $Y\in (\mathbb{R}^d)^N\setminus \Delta_N$,
\[
\nabla F_N(Y) = ( \, \mr{bary}_\rho (L_i^*(v,Y))\, )_i\,,
\]
where $\nabla$ denotes the gradient with respect to the $l^2$-inner product weighted by ${v}$, $\langle \cdot, \cdot\rangle_{v}$, and $(L_i^*(v,Y))_i$ is the unique Laguerre tessellation with generators $Y$ and volumes $v$.
Moreover, for all $B\in (\mathbb{R}^d)^N$,
\begin{equation}\label{eq:fnstar}
F^*_N(B) = \iota_{\mc{C}_N} (B) \coloneqq \left\{ \begin{array}{ll}
0 & \text{if } \nu_N(v,B) \preceq_C \rho,\\
+\infty & \text{otherwise},
\end{array}\right.
\end{equation}
where $F^*_N$ is the Fenchel-Legendre transform of $F_N$, again with respect to $\langle \cdot, \cdot\rangle_{v}$. {Furthermore, for all $Y \in (\mathbb{R}^d)^N$, the subgradient of $F_N$ at $Y$ is given by
\begin{align}
\label{eq:subgradF}
\partial F_N(Y) & = \mr{argmax} \left\{ \langle Y,X\rangle_v \,;\, X\in \mc{C}_N\right\}
\\
\label{eq:subgradF2}
& = \{ 
B \in \mc{C}_N ~;~ F_N(Y) = \langle B,Y\rangle_v  \}\,.
\end{align}
}
\end{proposition}

{Note that, for $Y \in \Delta_N$, we give an explicit expression for an element of $\partial F_N(Y)$ below in equation \eqref{eq:BY}.}

\begin{proof}
The first part follows from Proposition 5.1 in \cite{leclerc2020lagrangian} or Theorem 4.2 of \cite{bourne2024inverting}, for example. We give a sketch of the proof for completeness. Given any two vectors $Z,Y \in(\mathbb{R}^d)^N\setminus \Delta_N$,
the coupling $\gamma \in \mc{P}(\mathbb{R}^d\times \mathbb{R}^d)$ defined by
\[
\int_{\mathbb{R}^d \times \mathbb{R}^d} \varphi(x,y) \, \ed \gamma(x,y) = \sum_{i} \int_{L_i^*(v,Y)} \varphi(x,z_i) \, \ed \rho(x),
\]
for any bounded continuous function $\varphi:\mathbb{R}^d\times \mathbb{R}^d\rightarrow \mathbb{R}$, belongs to $\Gamma(\rho, \nu_{N}(Z,v))$ and is admissible for the problem defining $F_N(Z)$. {In addition, the coupling
$\gamma^* \in \mc{P}(\mathbb{R}^d\times \mathbb{R}^d)$ defined by
\begin{equation}
\label{eq: optimal coupling}   \int_{\mathbb{R}^d \times \mathbb{R}^d} \varphi(x,y) \, \ed \gamma^*(x,y) = \sum_{i} \int_{L_i^*(v,Y)} \varphi(x,y_i) \, \ed \rho(x)
\quad \forall \; \varphi \in C_b(\mathbb{R}^d\times \mathbb{R}^d)
\end{equation}
is the unique solution of the problem defining $F_N(Y)$.} In particular, this implies that
\[
F_N(Z) \geq F_N(Y) + \sum_i \int_{L_i^*(v,Y)} \langle x, z_i -y_i\rangle \, \ed \rho(x).
\]
This means that $B(Y) \coloneqq (\mr{bary}_\rho L_i^*(v,Y))_i \in \partial F_N(Y)$,  the subgradient of $F$ with respect to the metric $\langle \cdot,\cdot\rangle_v$, and therefore $F$ is convex. By stability of optimal transport, $\mr{bary}_\rho(L_i^*(v,Y))$ is a continuous function of $Y$ (see Proposition 5.1 in \cite{leclerc2020lagrangian}), which implies at once that $B(Y)=\nabla F_N(Y)$  and that $F$ is $C^1$. {On the other hand, by a similar reasoning, one can show that if $Y\in \Delta_N$, then $\partial F_N(Y)$ does not consist of a single vector so $F_N$ is not differentiable at these points}.

For the second part of the statement, observe that, by disintegrating the plan $\gamma$ in equation \eqref{eq:FY} with respect to its second marginal, the function $F_N$  can be equivalently written as follows:
\begin{equation}\label{eq:FN2}
F_N(Y) = \max \left\{ \sum_{i=1}^N  \int_{\mathbb{R}^d} \langle x, y_i\rangle v_i \ed \theta_i(x) ~~ ;~~ \theta_i \in \mc{P}_1(\mathbb{R}^d)\,, ~\sum_{i=1}^N \theta_i v_i = \rho \right\}\,.
\end{equation}
In view of Strassen's theorem (see in particular equation \eqref{eq:decomposition}), this is equivalent to 
	\begin{equation}\label{eq:support}
	F_N(Y) = \max\{ \langle Y, X\rangle_v \; ; \; X \in \mc{C}_N\}\,,
	\end{equation}
which in turn is equivalent to \eqref{eq:fnstar}.

{Next we prove \eqref{eq:subgradF}. 
Since $F_N = \iota_{\mc{C}_N}^*$, by Fenchel duality,
\[
\begin{aligned}
B \in \partial F_N(Y)& \iff Y \in \partial \iota_{\mc{C}_N}(B) \\& \iff \iota_{C_N}(X) \geq \iota_{\mc{C}_N}(B) + \langle Y,X-B\rangle_v\,, ~ \forall\, X\in(\mathbb{R}^{d})^N\\
& \iff  B\in \mc{C}_N~\text{ and } ~  \langle Y, B \rangle_v \geq \langle Y , X \rangle_v \,, ~ \forall\, X\in \mc{C}_N\,,
\end{aligned}
\]
which proves \eqref{eq:subgradF}. Finally, \eqref{eq:subgradF2} follows from \eqref{eq:subgradF} and \eqref{eq:support}.
}
\end{proof}

The function $F_N$ allows us to provide a characterization of the exposed and extreme points of $\mc{C}_N$. We remark that a similar result but in the continuous setting was established in \cite{ciosmak2023applications}, but considering the linear structure on probability measures, which is different from the type of convexity considered in this work.

We recall that the exposed faces of $\mc{C}_N$ are the sets of maximizers of a linear form on $\mc{C}_N$. Exposed faces are faces, i.e., convex subsets $\mc{F}\subset\mc{C}_N$ containing every closed segment in $\mc{C}_N$ whose interior intersects $\mc{F}$. Proper (exposed) faces are (exposed) faces that are neither the empty set nor $\mc{C}_N$. Exposed points are exposed faces containing a single point.

\begin{theorem}[Exposed points and faces] \label{th:boundary}
The following holds:
\begin{equation}\label{eq:envelopebary}\mc{C}_N =  \mathrm{conv}(\overline{\mc{B}^N_\mr{Lag}} )\,.\end{equation}
Moreover,
\begin{enumerate}
\item $\mc{B}^N_{\mr{Lag}}$ coincides with the set of exposed points of $\mc{C}_N$; 
\item the proper exposed faces of $\mc{C}_N$ have dimension $d(N-M)$ for $M=2,\ldots,N$. For a given $M$, the exposed faces of dimension $d(N-M)$ are given by $\{
{\mc{F}^N_\mr{Lag}}
(Y,\sigma)\}_{Y,\sigma}$  where $Y\in (\mathbb{R}^d)^M\setminus \Delta_M$, $\sigma:\{1,\ldots,N\} \rightarrow \{1,\ldots,M\}$ is any surjective map and
\begin{equation}\label{eq:Fnlag}
{\mc{F}^N_{\mr{Lag}}}
(Y,\sigma) = \Big\{ B \in (\mathbb{R}^d)^N~;~
\sum_{i\in \sigma^{-1}(j)}   v_{i} \delta_{b_{i}} \preceq_C \rho|_{L_j^*(w,Y)}\,,~\forall\, j = 1,\ldots,M \Big\},
\end{equation}
where $(L_i^*(w,Y))_i$ is the unique Laguerre tessellation with generators $Y$ and volumes $w = (w_1,\ldots, w_M)$, where
\begin{equation}\label{eq:wdef}
w_j = \sum_{i
\in \sigma^{-1}(j)} v_i \,, \quad \forall\,j = 1,\ldots,M\,;
\end{equation}
\item the relative boundary of $\mc{C}_N$ can be characterized as follows:
\[
\mathrm{rel} \,\partial \mc{C}_N = \bigcup_{ Y,\sigma }  
{\mc{F}^N_{\mr{Lag}}}
(Y, \sigma)
\]
where the 
union
is over any $Y \in (\mathbb{R}^d)^M\setminus \Delta_M$, $M=2,\ldots, N$ and $\sigma :\{1,\ldots,N\} \rightarrow \{1,\ldots,M\}$ surjective.
\end{enumerate}
\end{theorem}

\begin{proof}
First we prove \eqref{eq:envelopebary}. By \eqref{eq:gammastar}, we observe that
\[
F_N(Y) = \sup\left\{ \sum_i \int_{L_i} \langle x,y_i\rangle \,\ed \rho(x) \; ; \; (L_i)_i \in \mc{T}^N_{\mr{Lag}}(v,\rho)\right\} \,.
\]
Hence,
\[
F_N(Y) = \sup\left\{ \langle B,Y\rangle_{v} ~~;~~B \in \mc{B}^N_\mr{Lag} \right\}
,
\] 
and the right-hand side is precisely $\iota_{\mc{B}^N_{\mr{Lag}}}^*(Y)$, so that
$F^*_N= \iota_{\mc{B}^N_{\mr{Lag}}}^{**}$. 
But $F^*_N = \iota_{\mc{C}_N}$ by Proposition \ref{prop:dualF}. Therefore
$\mc{C}_N = \mathrm{conv}(\overline{\mc{B}^N_\mr{Lag}})$.

To prove the first point, observe that by Proposition \ref{prop:dualF},
\[
\begin{aligned}
B \in \mc{B}^N_{\mr{Lag}} &\iff \exists\, Y ~:~ \{B\} =\partial F_N(Y)\\
&\iff \exists\, Y ~:~ \{B\}= \mathrm{arg}\max \left\{ \langle X,Y\rangle_{v} ~~;~~X \in \mc{C}_N \right\},
\end{aligned}
\]
which is equivalent to $B$ being an exposed point of $\mc{C}_N$. Note that equation \eqref{eq:envelopebary} can be alternatively deduced as a direct consequence of this fact.

To prove the second point, observe that the  exposed faces of $\mc{C}_N$ can be obtained as the intersection of $\mc{C}_N$ with any supporting hyperplane different from the affine hull of $\mc{C}_N$. Any such hyperplane has a normal $\tilde{Y} \in (\mathbb{R}^d)^N$ which must be different from the constant vector (defined by $\tilde{y}_i=c$ for all $i$ and $c\neq 0$). 
Therefore, there exists $Y \in (\mathbb{R}^d)^M\setminus \Delta_M$,  for some $M\in \{2,\ldots, N\}$, and $\sigma :\{1,\ldots,N\} \rightarrow \{1,\ldots,M\}$ surjective, such that $\tilde{Y} = (\tilde{y}_i)_i \in (\mathbb{R}^{d})^N$ 
is given
by $\tilde{y}_i = y_{\sigma(i)}$ for all $i$. By equation \eqref{eq:subgradF2}, 
the intersection of $\mc{C}_N$ with the supporting hyperplane of outward normal $\tilde{Y}$ is the set of points $B\in\mc{C}_N$ such that $F_N(\tilde{Y}) = \langle B, \tilde{Y}\rangle_v$.

Let $(L^*_i(w,Y))_i$ be the Laguerre tessellation with generators $Y$ and volumes $w$, defined in \eqref{eq:wdef}.
Let us take $\gamma \in \Gamma(\rho,\nu_M(w,Y))$  
to be the optimal transport plan
defined by
\[
\ed \gamma(x,y) = \sum_{i=1}^N v_i \ed \theta_i(x) \otimes \delta_{\tilde{Y}_i}(y),
\]
where $\theta_i \in \mc{P}(\mathbb{R}^d)$ satisfies 
\begin{equation}\label{eq:thetasum}
\sum_{i\in \sigma^{-1}(j)} v_i \theta_i = \rho|_{L_j^*(w,Y)}\,,
\end{equation}
for all $j=1,\ldots, M$. By construction (see Section \ref{subsec: SDOT}), 
\[
F_M(Y) = \int \langle x,y\rangle \ed \gamma(x,y)\,.
\]
Therefore,
\[
F_N(v,\rho;\tilde{Y}) = F_M(w,\rho;Y) = \sum_{i=1}^N  v_i \left\langle \int x \ed \theta_i(x) , \tilde{y}_i \right\rangle\,.
\]
This means that the vectors $B = (b_1,\ldots,b_N) \in\mc{C}_N$ belonging to the exposed face with normal $\tilde{Y}$ are precisely those satisfying
\begin{equation}
\label{eq: barycenter condition}
b_i =\int x\ed \theta_i(x)\,,
\end{equation}
for some collection of probability measures $(\theta_i)_i$ as above.
Finally, by Strassen's Theorem (see equation \eqref{eq:decomposition}), the existence of probability measures $\theta_i$ satisfying \eqref{eq:thetasum}
and \eqref{eq: barycenter condition}
is equivalent to the convex order relation 
\[
\sum_{i\in \sigma^{-1}(j)}   v_{i} \delta_{b_{i}} \preceq_C \rho|_{L_j^*(w,Y)}.
\]
This proves equation \eqref{eq:Fnlag}. Finally, by direct computation and Lemma \ref{lem:dimension} (see also Remark \ref{rem:facesoffaces} below),
\[
\mr{dim}(\mc{F}^N_{\mr{Lag}}(Y,\sigma)) = \sum_{j=1}^M d(\#\sigma^{-1}(j) -1) = d(N-M)\,.
\]

The third point follows from the fact that the relative boundary of a convex body is the union of all its proper exposed faces. 
\end{proof}

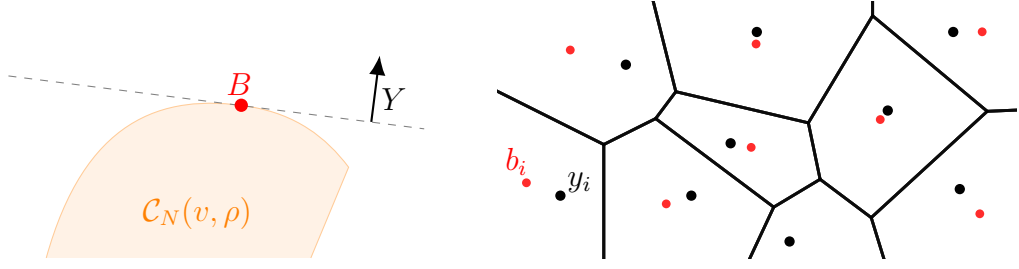
\begin{figure}
\begin{tikzpicture}[scale=1.]

% --- Convex set (filled with smooth shading) ---
\fill[color=orange!10] 
  (-2.5,-.6)
  .. controls (-2,1.2) and (-1,1.6) .. (0.2,1.4)
  .. controls (0.8,1.3) and (1.2,1.0) .. (1.5,0.6)
  -- (1.0,-.6)
  -- (-2.5,-.6)
  -- cycle;

\draw[color=orange!40] 
  (-2.5,-.6)
  .. controls (-2,1.2) and (-1,1.6) .. (0.2,1.4)
  .. controls (0.8,1.3) and (1.2,1.0) .. (1.5,0.6)
  -- (1.0,-.6);

% --- Tangent/supporting line ---
\draw[color =gray, dashed]
  (-3.,1.808) -- (2.5,1.108);

% --- Contact point ---
\fill[red] (0.08,1.42) circle (2.5pt);
\node[color =red] at (0.05,1.7) {$B$};

\node at (2.1,1.5) {$Y$};

\node[color= orange] at (-.5,0) {$\mc{C}_N(v,\rho)$};

\draw[-{Latex[length=2.6mm,width=2.2mm]}, line width = .8pt]  (1.8,1.2) -- (1.912,2.08);
\end{tikzpicture}\hspace{2em}
\begin{overpic}[scale=.55, trim=9 9 9 9, clip]{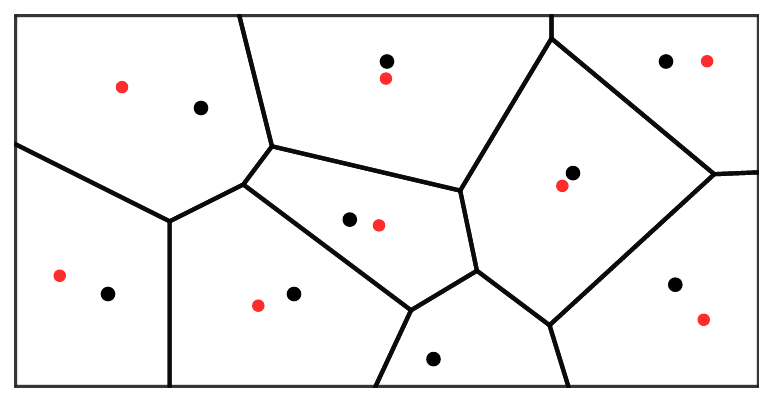}
				\put(13.5,14){$y_i$}
				\put(1.5,17){{\color{red} $b_i$}}
\end{overpic}
\vspace{-2em}
\caption{The vector of barycenters $B =(b_i)_i$ of a Laguerre tessellation with generators $Y=(y_i)_i$ and volumes $v=(v_i)_i$ (right figure) is the unique maximizer of $X \mapsto \langle X,Y\rangle_v$ on $\mc{C}_N(v,\rho)$ (left). }\label{fig:intuition}
\end{figure}

Figure \ref{fig:intuition} provides a graphical illustration of the first point in Theorem \ref{th:boundary} characterizing the exposed points of $\mc{C}_N$. 

\begin{example}
Here are two simple examples of $\mc{C}_N$, which were constructed by finding $\mc{B}^N_\mr{Lag}$
and using equation \eqref{eq:envelopebary}. 
\begin{enumerate}
    \item If $d=1$, then $\mc{C}_N$ is a convex polytope. In particular, there is no distinction between exposed and extreme points in this case. When all the volumes $v_i$ are equal, then $\mc{C}_N$ is exactly the permutahedron associated with the barycenters $b_i$ of $L_i$, where $(L_i)_i$ the unique partition of the real line $\mathbb{R}$ by disjoint intervals such that $\rho(L_i) = 1/N$.
    \item If $N=d=2$, $v_1=v_2=\frac 12$ and $\rho = \frac{1}{\pi} \chi_{B(0,1)}$ is the uniform probability distribution on the unit disk, then $\mc{C}_N(v,\rho) \subset \mathbb{R}^4$ is a two-dimensional disk. It lies in the plane spanned by $(1,0,-1,0), \,(0,1,0,-1)$, has center $0$ and radius $r=\frac{4}{3 \pi^2}$, where $(r,0)$ is the barycenter (with respect to $\rho$) of the half-disk $B(0,1) \cap \{x_1 \ge 0\}$.  
\end{enumerate}
\end{example}

\begin{remark}\label{rem:facesoffaces} Since each exposed face $\mc{F}^N_{\mr{Lag}}(Y,\sigma)$ can itself be described  by convex order relations analogous to those defining $\mc{C}_N(v,\rho)$ (see equation \eqref{eq:Fnlag}), Theorem \ref{th:boundary} also provides a precise characterization of the exposed faces of $\mc{F}^N_{\mr{Lag}}(Y,\sigma)$. Applying the theorem recursively this yields a hierarchical description of the boundary of $\mc{C}_N(v,\rho)$. To make this precise, 
using the same notation as in Theorem \ref{th:boundary},
 let us denote, for any $j=1,\ldots,M$, 
\begin{gather*}
\rho^j \coloneqq \frac{\rho|_{L^*_j(w,Y)}}{w_j}\,,  \quad \text{where} \quad w = (w_j)_{j} \in \mathbb{R}^{M}_{>0} \quad \text{with}\quad w_j \coloneqq \sum_{i\in \sigma^{-1}(j)} v_i\,,\\
N^j \coloneqq \# \sigma^{-1}(j)\,, \quad v^j \coloneqq \left(\frac{v_{i_1}}{w_j},\ldots,\frac{v_{i_{N^j}}}{w_j}\right)\,, \quad \text{where}\quad i_k\in \sigma^{-1}(j) ~\text{ and }  ~i_k<i_{k+1}\,.
\end{gather*}
Moreover,  for any $X\in (\mathbb{R}^d)^N$, we denote
\[
X^j \coloneqq \left(x_{i_1},\ldots, x_{i_{N^j}}\right) \in (\mathbb{R}^d)^{N^j} ~\quad \text{ where } ~i_k\in \sigma^{-1}(j) ~\text{ and }  ~i_k<i_{k+1}\,.
\]
Then, by the characterization of the face  $
\mc{F}
^N_{\mr{Lag}}(Y,\sigma)$ in Theorem \ref{th:boundary}, 
\begin{equation}\label{eq:characterisationF}
X =(x_1, \ldots, x_N) \in \mc{F}^N_{\mr{Lag}}(Y,\sigma) \iff X^j \in \mc{C}_{N^j}(v^j,\rho^j)\,, \quad \forall \, j=1,\ldots,M\,.
\end{equation}
In other words, the exposed face $\mc{F}
^N_{\mr{Lag}}(Y,\sigma)$ can be identified with  the Cartesian product $ \times_{j=1}^M \mc{C}_{N^j}(v^j,\rho^j)$.
\end{remark}

Combining Proposition \ref{prop:uniqueness} and Theorem \ref{th:boundary} yields the following result on extreme points. Before stating it, we recall that an extreme point of $\mc{C}_N$ is a point $B\in\mc{C}_N$ that cannot be expressed as the linear combination of two other points in $\mc{C}_N$ different from $B$.

\begin{theorem}[Extreme points] \label{th:extreme} $\mc{B}^N_{\mr{hLag}} = \overline{\mc{B}^N_{\mr{Lag}}}$ is the set of extreme points of $\mc{C}_N$.
\end{theorem}
\begin{proof}
By Lemma \ref{lem:closurebary} and Theorem \ref{th:boundary}, $\mc{B}^N_{\mr{hLag}}(v,\rho) = \overline{\mc{B}^N_{\mr{Lag}}(v,\rho)}$ is the closure of the exposed points of $\mc{C}_N(v,\rho)$. Hence, $\mc{B}^N_{\mr{hLag}}(v,\rho)\subset \mr{rel}\,\partial\mc{C}_N(v,\rho)$ and it must contain all the extreme points of $\mc{C}_N(v,\rho)$ by Straszewicz's Theorem \cite{straszewicz1935exponierte}, which states that the set of exposed points of a compact convex set is dense in the set of its extreme points.

It remains to prove the converse inclusion. Let
\(B=(b_1,\ldots,b_N)\in \mc{B}^N_{\mr{hLag}}\), and let
\((L_i)_{i=1}^N\in \mc{T}^N_{\mr{hLag}}(v,\rho)\) be the corresponding hierarchical Laguerre tessellation. Thus
\[
b_i = \frac{1}{v_i}\int_{L_i} x\,\ed\rho(x),
\qquad
\rho(L_i)=v_i.
\]
By Proposition~\ref{prop:uniqueness}, the collection
$\bar\theta_i \coloneqq \frac{1}{v_i}\rho|_{L_i}$ is uniquely characterized by the conditions
\begin{equation} \label{eq:thetai}
\sum_{i=1}^N v_i\bar\theta_i=\rho,
\qquad
\int_{\mathbb{R}^d} x\,\ed\bar\theta_i(x)=b_i\,.
\end{equation}
Suppose now that \(B\) is not extreme. Then there exist
$X=(x_1,\ldots,x_N),Z=(z_1,\ldots,z_N)\in \mc{C}_N$ and some
$\lambda\in(0,1)$ such that $B=(1-\lambda)X+\lambda Z$.
By Strassen's Theorem, applied in the form of \eqref{eq:decomposition}, there exist probability measures
$(\theta_i^X)_i$ and $(\theta_i^Z)_i$ such that
\[
\begin{aligned}
\sum_{i=1}^N v_i\theta_i^X=\rho, 
\qquad \int_{\mathbb{R}^d} x\,\ed\theta_i^X(x)=x_i, \\
\sum_{i=1}^N v_i\theta_i^Z=\rho,
\qquad \int_{\mathbb{R}^d} x\,\ed\theta_i^Z(x)=z_i.
\end{aligned}\]
Define $\theta_i \coloneqq (1-\lambda)\theta_i^X+\lambda\theta_i^Z.$
Then each \(\theta_i\) is a probability measure, and it satisfies  \eqref{eq:thetai}. Hence, by uniqueness, $\theta_i=\bar\theta_i$ for all $i$. We now show that this forces \(X=Z=B\). Since
$(1-\lambda)\theta_i^X+\lambda\theta_i^Z=\bar\theta_i$,
and all measures involved are nonnegative, and since $\lambda \neq 1$, this implies that $\theta_i^X$ is absolutely continuous with respect to \(\bar\theta_i\). Let $a_i^X$ be the density of $\theta_i^X$ with respect to $\bar\theta_i$, so that $v_i\theta_i^X = a_i^X\,\rho|_{L_i}$. Then, using  $\sum_i v_i\theta_i^X=\rho$, we obtain
$$
\rho
=
\sum_{i=1}^N a_i^X\,\rho|_{L_i}.
$$
Since the cells \(L_i\) form a partition up to a \(\rho\)-negligible set, it follows that
$a_i^X$ coincide with the indicator function of $L_i$ $\rho$-a.e. This implies that for all $i$, $
\theta_i^X=\bar\theta_i$. Thus,
$$ x_i = \int_{\mathbb{R}^d} x\,\ed\theta_i^X(x) = \int_{\mathbb{R}^d} x\,\ed\bar\theta_i(x) = b_i. $$
This implies that $X=B$ and similarly that $Z = B$, thus showing that $B$ is extreme.
\end{proof}

\subsection{Irreducible Laguerre tessellations and maximal faces} 
We say that a Laguerre tessellation $(L_i)_{i=1}^N$ is \emph{irreducible} if there does not exist any other Laguerre tessellation $(\tilde{L}_j)_{j=1}^M$ with $1<M<N$ whose cells $\tilde{L}_j$ can be expressed as a union of those from $(L_i)_{i=1}^N$; see Figure \ref{fig:irreducible}. We denote by $\mc{T}^N_{\mr{iLag}}(v,\rho)$ the set of irreducible tessellations in $\mc{T}^N_{\mr{Lag}}(v,\rho)$. In this section, we study the set of barycenters of such tessellations defined as follows:
\[
\mc{B}^N_{\mr{iLag}}(v,\rho) \coloneqq \{ B \in (\mathbb{R}^{d})^N ~;~ \exists\, (L_i)_i \in \mc{T}^N_{\mr{iLag}}(v,\rho) \text{ such that } b_i = \mathrm{bary}_{\rho}(L_i) ~ \forall \, i\}\,.
\]

Specifically, we show the following result:
\begin{theorem}\label{th:interior} 
Let $\exp(\mc{C}_N(v,\rho))$ denote the set of exposed points of $\mc{C}_N(v,\rho)$. 
Then $\mc{B}^N_{\mr{iLag}}(v,\rho)$ is the interior of $\mc{B}^N_{\mr{Lag}}(v,\rho)= \exp(\mc{C}_N(v,\rho)) $ in  $\mr{rel}\, \partial \mc{C}_N(v,\rho)$.
\end{theorem}

Before proving Theorem \ref{th:interior}, we describe the relation between $\mc{B}^N_{\mr{iLag}}$ and the convex set $\mc{C}_N$. First, we recall that a maximal face of a convex set is a proper face that is maximal under inclusion, i.e., it is not a strict subset of any other proper face. 

\begin{proposition}\label{prop:maximal} The following are equivalent:
\begin{enumerate}
    \item $B \in \mc{B}^N_{\mr{iLag}}(v,\rho)$;
    \item $\{B\} \subset \mc{C}_N(v,\rho)$ is a maximal face.
\end{enumerate}
\end{proposition}
\begin{proof}
    $\{B\}$ is a maximal face if and only if $B$ is an exposed point that does not belong to any proper exposed face other than $\{B\}$. The result then follows as a direct consequence of Theorem \ref{th:boundary} and Proposition \ref{prop:uniqueness}. More precisely,  
    suppose that $\{B\}$ is a maximal face and $B \notin \mc{B}^N_{\mr{iLag}}(v,\rho)$. Then, by definition of $\mc{B}^N_{\mr{iLag}}$, there exists $M<N$, $Y \in (\mathbb{R}^d)^M \setminus \Delta_M $ and $\sigma : \{1,\ldots,N\} \to \{ 1, \ldots ,M\}$ such that $B \in \mc{F}^N_{\mr{Lag}}(Y,\sigma)$ (defined in equation \eqref{eq:Fnlag}). But then, by Theorem \ref{th:boundary},
    $B$ would be belong to an exposed face other than $\{B\}$. The converse holds by a similar argument 
    since, for any $B\in \mc{B}^N_{\mr{iLag}}(v,\rho)$, by Proposition \ref{prop:uniqueness}, there is a unique tessellation associated with it, which is in $\mc{T}^N_{\mr{iLag}}$ by construction.
\end{proof}

\begin{figure}
    \begin{tikzpicture}

% Draw the first square with thick lines
\fill[cyan!40,  fill opacity=0.5] (0,0) rectangle (3,3);
\draw[thick] (1.5,-.5) -- (1.5,3.5);
\draw[thick] (-.5,1.5) -- (3.5,1.5);

\fill[black] (.75,2.25) circle (2pt); % 
\node[above] at (.75,2.25) {$b_1$};
\fill[black] (2.25,2.25) circle (2pt); % 
\node[above] at (2.25,2.25) {$b_2$};
\fill[black] (2.25,.75) circle (2pt); % 
\node[above] at (2.25,.75) {$b_3$};
\fill[black] (.75,.75) circle (2pt); % 
\node[above] at (.75,.75) {$b_4$};

% Draw the second square with thick lines
\fill[cyan!40,  fill opacity=0.5] (5,0) rectangle (5+3,3);
\draw[thick] (5+1.5,-.5) -- (5+1.5,3.5);
\fill[red] (5+.75,1.5)circle (2pt); % 
\node[above,red] at (5+.75,1.5) {$y_1=y_4$};
\fill[red] (5+2.25,1.5) circle (2pt); % 
\node[above,red] at (5+2.25,1.5) {$y_2=y_3$};
% Connect the points with lines
% Draw the second square with thick lines
\fill[cyan!40, fill opacity=0.5] (10,0) rectangle (10+3,3);
\draw[thick] (10-.5,1.5) -- (10+3.5,1.5);
\fill[red] (10+1.5,2.25)circle (2pt); % 
\node[above,red] at (10+1.5,2.25) {$y_1=y_2$};
\fill[red] (10+1.5,.75)circle (2pt); % 
\node[above,red] at (10+1.5,.75) {$y_3=y_4$};
\end{tikzpicture}
    
    \vspace{-2em}
    \caption{The figure on the left shows an example of a Laguerre tessellation that is not irreducible. For this tessellation, the vector of barycenters $B = (b_1, b_2, b_3, b_4)$, corresponding to $\rho$ uniform on the square and $v_1=v_2=v_3=v_4$, lies on two distinct exposed faces of $\mc{C}_4(v,\rho)$, each of dimension strictly greater than zero. These faces are exposed by the hyperplanes with normal vectors $Y = (y_1, y_2, y_3, y_4)$ shown in the figure (center and right).}
    \label{fig:irreducible}
\end{figure}
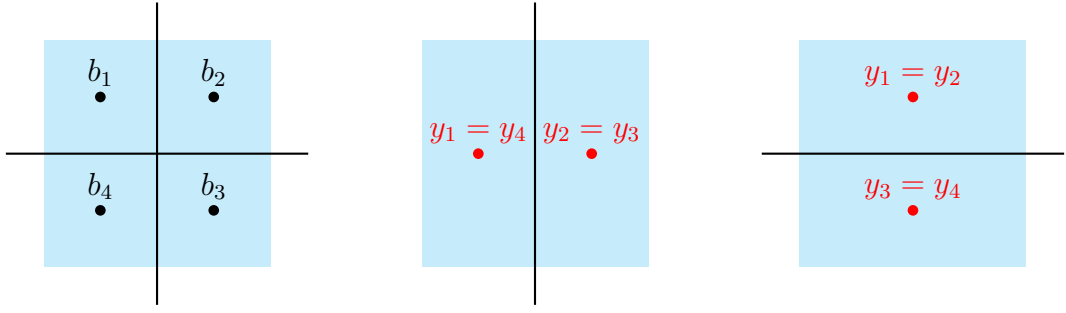

\begin{proof}[Proof of Theorem \ref{th:interior}] It suffices to show that $\mc{B}^N_{\mr{iLag}}(v,\rho)$ is open in $\mr{rel}\, \partial \mc{C}_N(v,\rho)$. In fact, by construction $\mc{B}^N_{\mr{iLag}}\subseteq \mc{B}^N_{\mr{Lag}}$, and by Proposition \ref{prop:maximal}, $\mc{B}^N_{\mr{Lag}}\setminus \mc{B}^N_{\mr{iLag}} \subseteq \partial \mc{B}^N_{\mr{Lag}}$, 
the boundary of $\mc{B}^N_{\mr{Lag}}$ in $\mr{rel}\, \partial \mc{C}_N(v,\rho)$.

Let us suppose that $\mc{B}^N_{\mr{iLag}}(v,\rho)$ is not open in $\mr{rel}\, \partial \mc{C}_N(v,\rho)$. Then, for some $B \in \mc{B}^N_{\mr{iLag}}(v,\rho)$, there exists a sequence of points $(B^n)_n \subset \mr{rel}\, \partial \mc{C}_N(v,\rho) \setminus \mc{B}^N_{\mr{iLag}}(v,\rho)$ such that $B^n \rightarrow B$ as $n\rightarrow \infty$. By Proposition \ref{prop:maximal} and Theorem \ref{th:boundary}, there exists a sequence of vectors $(Y^n)_n \subset \Delta_N \setminus\{Y = (y, \ldots,y)\; ;\; y\in \mathbb{R}^d\}$ such that
\[
B^n \in \mr{argmax}\{ \langle Y^n, X\rangle_v \;; \; X \in \mc{C}_N\}.
\]
We can suppose that each vector $Y^n$ is  normalized so that $y^n_1=0$ and $\max_{i\neq k} |y_i^n-y_k^n| =1$. Hence, up to the extraction of a subsequence, $Y^n \rightarrow Y^\infty \in \Delta_N \setminus\{Y = (y, \ldots,y)\; ;\; y\in \mathbb{R}^d\}$. Moreover,
\[
B\in \mr{argmax}\{ \langle Y^\infty, X\rangle_v \;; \; X \in \mc{C}_N\}.
\]
But then, reasoning as in the proof of Theorem \ref{th:boundary}, this means that $\{B\}$ cannot be a maximal face, which is a contradiction by Proposition \ref{prop:maximal}.
\end{proof}

\section{Semi-discrete Wasserstein projections in convex order} \label{sec:projection} In this section  we describe how the $L^2$-projection of a vector $X \in (\mathbb{R}^d)^N$ onto $\mc{C}_N(v,\rho)$ can be reinterpreted as a Wasserstein projection of the discrete measure $\nu_N(v,X)$ onto the set of measures $\mu \preceq_C \rho$. While this problem is of independent interest, it will be instrumental to our approach to 
reconstructing a Laguerre tessellation from its cell volumes and barycenters, discussed in the  following sections.

Let us define the following weighted orthogonal projection
\begin{equation}\label{eq:l2projection}
P_{\mc{C}_N}(X) \coloneqq \mr{argmin}\{ \| B- X\|_{v}^2 ~;~ B \in \mc{C}_N \}.
\end{equation}
Note that since $\mc{C}_N$ is a bounded, closed and convex set, this problem is well-defined and admits a unique solution. 

It will be useful to consider its dual (with respect to the duality expressed by the inner product $\langle \cdot, \cdot\rangle_v$), which 
can be written
as follows:

\begin{lemma}\label{lem:dual} Let $X\in (\mathbb{R}^d)^N$. The following holds:
\begin{equation}\label{eq:dualityproj}
\inf_{B {\in (\mathbb{R}^d)^N}} \left\{ \frac{\| B-X\|^2_{v}}{2} + \iota_{\mc{C}_N}(B)\right\} = - \inf_{Y {\in (\mathbb{R}^d)^N}} \left\{\frac{\| Y\|^2_{v}}{2} - \langle Y,X\rangle_{v} + F_N(v,\rho;Y)\right\}.
\end{equation}
Moreover the problems on the left- and right-hand sides admit  
unique solutions, denoted $P_{\mc{C}_N}(X)$ and $Y^*$ respectively, which 
{satisfy}
\begin{equation}\label{eq:ypcx}
Y^*  = X-P_{\mc{C}_N}(X)\,.
\end{equation}
\end{lemma}

 \begin{proof} 
The result is classical since the right-hand side of \eqref{eq:dualityproj} is just the Legendre transform of the sum $Y\mapsto F_N(Y) + \|Y\|^2_v/2$, which is given by the inf-convolution of the Legendre transforms of $F_N$ and $ \| \cdot\|^2_v/2$. We detail the proof for clarity.

Both problems admit a unique solutions by the strong convexity and lower-semicontinuity of the objective functions. Moreover, $B^*= P_{\mc{C}_N}(X)$ is the unique solution of the problem on the left-hand side of \eqref{eq:dualityproj} if and only if 
\begin{align}
\nonumber
0 \in \partial \left( B \mapsto \frac{\| B - X \|_v^2}{2} +  \iota_{\mc{C}_N}(B) \right) (B^*)
\quad & \Longleftrightarrow
\quad 
X - B^* \in \partial \iota_{\mc{C}_N}(B^*)
\\
\label{eq:opt condition 1}
& \Longleftrightarrow
\quad
B^* \in \partial F_N(X-B^*)\,,
\end{align}
since $F_N^* = \iota_{\mc{C}_N}$, by Proposition~\ref{prop:dualF}. Similarly, 
$Y^*$ is the unique solution of the problem on the right-hand side of \eqref{eq:dualityproj} if and only if 
\begin{equation}
\label{eq:opt condition 2}   
0 \in \partial \left( Y \mapsto
\frac{\| Y\|^2_{v}}{2} - \langle Y,X\rangle_{v} + F_N(Y)
\right) (Y^*)
\quad \Longleftrightarrow \quad
X - Y^* \in \partial F_N(Y^*).
\end{equation}
Then equation \eqref{eq:ypcx} follows from \eqref{eq:opt condition 1} and \eqref{eq:opt condition 2}, and moreover 
$B^* \in \partial F(Y^*)$. Equation \eqref{eq:dualityproj} follows since, by Fenchel duality and equation \eqref{eq:ypcx},
\[
F_N(Y^*) + \iota_{\mc{C}_N}(B^*) = \langle B^*, Y^*\rangle_v = -\frac{\| B^*-X\|^2_v}{2} - \frac{\|Y^*\|^2_v}{2} + \langle Y^*,X\rangle_v\,.
\]
\end{proof}

We now show the equivalence between problem \eqref{eq:l2projection} and 
the
2-Wasserstein projection problem \eqref{eq:projection}. In order to state this, let us first recall the definition of the 2-Wasserstein distance $W_2$: 
\begin{equation}\label{eq:w2}
W_2^2(\mu,\nu) \coloneqq \inf_{\gamma \in \Gamma(\mu,\nu)}\left\{ \int_{\mathbb{R}^d \times \mathbb{R}^d} 
|x-y|^2 \, \ed \gamma(x,y) \right\}
\end{equation}
for any $\mu,\nu \in \mc{P}_2(\mathbb{R}^d)$ probability measures with finite second moments, and 
where $\Gamma(\mu,\nu) \in \mc{P}(\mathbb{R}^d\times\mathbb{R}^d)$ is the set of probability measures with first and second marginals equal to $\mu$ and $\nu$, respectively.

\begin{proposition}\label{prop:w2proj} Let $X\in (\mathbb{R}^d)^N$. The measure 
$\nu_N(v,P_{\mc{C}_N}(X))$
is the unique solution 
of
the projection problem
\begin{equation}\label{eq:projW2}
\inf \{ W^2_2(\nu,\nu_N(v,X))\,;\, \nu \preceq_C \rho\}.
\end{equation}
\end{proposition}

\begin{proof} Existence and uniqueness of the minimizer $\nu^*$ of \eqref{eq:projW2} has been proven in \cite{gozlan2020mixture}. In the same work, the authors showed that 
 there exists a Lipschitz continuous
optimal transport map $T^*$ from $\nu_N(v,X)$ to $\nu^*$.
This implies that 
$\nu^* = \nu_N(v,Y^*)$
for some vector $Y^* =(T^*(x_i))_i \in(\mathbb{R}^d)^N$. 
Using the fact that
\[
W_2^2(\nu_N(v,Y),\nu_N(v,X))\leq 
\| Y-X\|_v^2,
\]
 for any $Y\in (\mathbb{R}^d)^N$, we obtain
\[
\|Y^* - X\|_v^2 = \inf \{ W^2_2(\nu,\nu_N(v,X))\,;\, \nu \preceq_C \rho\}\leq 
\inf \{ \|Y - X\|_v^2 \,;\, \nu_N(v,Y) \preceq_C \rho\},
\]
which implies that $Y^* =P_{\mc{C}_N}(X)$.
\end{proof}

We conclude this section by remarking that the set $\mc{C}_N$ gives rise to an alternative characterisation of the quantization problem and optimal centroidal Laguerre tessellations.
\begin{remark}[Connection with optimal quantization] 
Consider the non-convex optimization problem
\begin{equation}
\label{eq:quantization1}
\inf_{X \in (\mathbb{R}^d)^N} \left\{-\frac{\|X\|^2_v}{2} \; ;\; X \in \mc{C}_N \right\}.
\end{equation}
Suppose that $X \in (\mathbb{R}^d)^N \setminus \Delta_N$ is a global minimizer. Then $0 \in \partial(-\| \cdot \|_v^2/2 + \iota_{\mc{C}_N})(X)$, and it follows from Proposition \ref{prop:dualF} that $X \in \partial F_N(X)$ or equivalently $x_i = \mr{bary}_{\rho}(L_i^*(v,X))$ for all $i$. In other words, $X$ generates a centroidal Laguerre tessellation \cite{BourneRoper2015,LevyCentroidalPower}. Moreover, it can be shown using Toland duality \cite{Toland79} that global minimizers of problem \eqref{eq:quantization1} are global minimizers of the following optimal quantization problem \cite{BKRS20,MerigotSantambrogioSarrazin, PortalesCazellesPauwels2025}:
\begin{equation}
\label{eq:quantization2}    
     \inf_{Y \in (\mathbb{R}^d)^N} W_2(\rho,\nu_N(v,Y)).
\end{equation}
Here we look for the best approximation of the absolutely continuous measure $\rho$ by a discrete measure with fixed volumes.
Problem \eqref{eq:quantization1} provides an alternative formulation of the quantization problem \eqref{eq:quantization2}. For an analogous statement in the fully discrete setting, where $\rho$ is discrete, see \cite[Theorem 2.4]{brieden2012optimal}.
\end{remark}

\section{Laguerre tessellation reconstruction}\label{sec:laguerre}
In Theorem \ref{th:boundary}, we showed that the vectors in $(\mathbb{R}^d)^N$  
that
can be obtained as the barycenters of Laguerre tessellations with given volumes, $\mc{B}_{\mr{Lag}}^N$, are precisely the exposed points of $\mc{C}_N$. 
Moreover, for any $B\in \mc{B}_{\mr{Lag}}^N$, there exists a supporting hyperplane with normal $Y {\in (\mathbb{R}^d)^N}\setminus \Delta_N$ that exposes $B$, i.e., such that \[\{B\} = \mr{argmax} \{ \langle X, Y\rangle_v \,;\, X \in \mc{C}_N\}.\]
By Proposition \ref{prop:dualF} the vector $Y$ is a vector of generators associated with the tessellation with barycenters $B$.

We can then reformulate the Laguerre tessellation reconstruction and fitting problems as follows: can one retrieve an exposed point of $\mc{C}_N$ that is closest to a given point $B \in (\mathbb{R}^d)^N$ and recover a normal vector $Y$ that exposes it? 
If 
$B$
is already an exposed point, we can regard this as a reconstruction problem (as this is equivalent to Question \ref{ques:reconstruction}), 
{otherwise we regard it as a fitting problem.}
In this section we {focus on the reconstruction problem, and discuss two possible methods to solve this}.

\subsection{Reconstruction via Legendre transform} 
By Proposition \ref{prop:dualF}, given a point $B \in \mc{B}^N_{\mr{Lag}}$, the generators $X$ associated with $B$ solve the problem
\begin{equation}
\label{eq:legF}
\sup_{Y}\, H(Y;B) \,,\quad H(Y;B) \coloneqq \langle Y,B\rangle_v - F_N(Y)\,.
\end{equation}
In \cite{bourne2024inverting} it was proposed to solve the reconstruction problem precisely by computing a solution to this maximization problem. 
While this seems to work well in practice (see \cite[Section 5.1]{bourne2024inverting}),  the function 
$H(\, \cdot \,; B)$ is not strictly concave and it
is also maximized by the constant vector $Y=(y,\ldots,y) \in (\mathbb{R}^d)^N$ for any $y \in \mathbb{R}^d$, and so maximising $H$ is not guaranteed to solve the reconstruction problem. We remark that this issue can be addressed by adding a convex constraint to the problem. Specifically, fixing arbitrary indices $i\neq j$ and $\delta>0$, one can restrict the maximization over the set
\[
\{Y\in(\mathbb{R}^d)^N\,;\,\langle b_i - b_j, y_i-y_j\rangle\geq \delta\,\}.
\]
In fact, by the invariance of the barycenters upon dilation of the generators, it is always possible to construct a set of generators $Y$ associated with $B$ that satisfies this constraint. 
On the other hand, by Proposition \ref{prop:uniqueness}, any maximizer $Y$  satisfying the constraint must also satisfy $y_i \neq y_j$ for all $i\neq j$, and is therefore a set of generators associated with the barycenters $B$.

\subsection{Reconstruction via projections in convex order} We describe here a second approach to 
solving
the reconstruction problem based on the orthogonal projection on $\mc{C}_N$ discussed in Section \ref{sec:projection}, which has the advantage of being unconstrained and strongly convex.

The idea is illustrated in Figure \ref{fig:intuition2}. Given an exposed point $B$ of $\mc{C}_N$, we first perturb it by a dilation, i.e.\ we consider the point $\lambda B$, with $\lambda >1$. One can check that $\lambda B \notin \mc{C}_N$ and that $B^* = P_{\mc{C}_N}(\lambda B)$ is on the relative boundary of $\mc{C}_N$ (see Lemma \ref{lem:projection} below). By standard properties of the orthogonal projection, $B^*$ is exposed by the vector $Y^* =\lambda B - B^*$, and $\|B-B^*\|_v \leq |\lambda -1| \|B\|_v$ is small for $\lambda$ close to one. Finally, if $Y^* \notin \Delta_N$, we have found an approximate solution of the reconstruction problem, since then $Y^*$ is a vector of distinct generators associated with $B^*$.

In order to show in which cases the property $Y^* \notin \Delta_N$ is satisfied we will need the following lemma:

\begin{figure}
\centering
\begin{tikzpicture}[scale=1.]

% --- Convex set (filled with smooth shading) ---
\fill[color=orange!10] 
  (-2.5,-.6)
  .. controls (-2,1.2) and (-1,1.6) .. (0.2,1.4)
  .. controls (0.8,1.3) and (1.2,1.0) .. (1.5,0.6)
  -- (1.0,-.6)
  -- (-2.5,-.6)
  -- cycle;

\draw[color=orange!40] 
  (-2.5,-.6)
  .. controls (-2,1.2) and (-1,1.6) .. (0.2,1.4)
  .. controls (0.8,1.3) and (1.2,1.0) .. (1.5,0.6)
  -- (1.0,-.6);

% --- Tangent/supporting line ---
\draw[color =gray, dashed]
  (-3.,1.808) -- (2.5,1.108);

\draw[color =gray, dashed]
  (-3,2.5) -- (2.5,.66);

% --- Contact point ---
\fill[red] (0.08,1.42) circle (2.5pt);

\fill[gray!80] (0.55,1.3) circle (2.5pt);

\node[gray!80] at (0.55,1.) {$B^*$};

\draw[-{Latex[length=2.6mm,width=2.2mm]}, line width = .8pt, color=gray!80]  (0.55,1.3)  -- (0.55+1.84*.15,1.3+5.5*.15 )
;
\node[color=gray!80] at (1.1,1.7) {$Y^*$};

\fill[red] (0.55+1.84*.16,1.3+5.5*.16 )  circle (2.5pt)
node[above] {$\lambda B$}
;

\node[color =red] at (0.05,1.7) {$B$};

\node at (2.1,1.5) {$Y$};

\node[color= orange] at (-.5,0) {$\mc{C}_N(v,\rho)$};

\draw[-{Latex[length=2.6mm,width=2.2mm]}, line width = .8pt]  (1.8,1.2) -- (1.912,2.08);
\end{tikzpicture}
\vspace{-2em}
\caption{In order to approximate $Y$, the direction exposing $B$, we compute $Y^* = \lambda B - B^*$ where $B^* = P_{\mc{C}_N(v,\rho)} (\lambda B)$ for $\lambda>1$ close to 1. By construction, $B^*$ maximizes $\langle \cdot,Y^*\rangle_v$ on $\mc{C}_N$, i.e., $Y^*$ exposes $B^*$, and $B^*\rightarrow B$ as $\lambda\rightarrow 1$. }\label{fig:intuition2}
\end{figure}
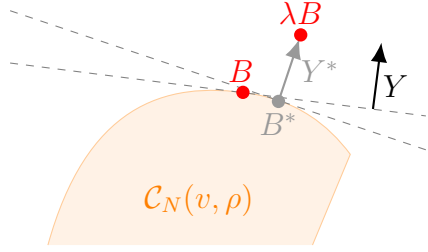

\begin{lemma}\label{lem:projection}
    Let $N\geq 2$, and suppose that $B$ is in  the relative boundary of $\mc{C}_N$, denoted $\mr{rel}\,\partial \mc{C}_N$. Then, for any $\lambda >1$ we also have $P_{\mc{C}_N} (\lambda B)\in \mr{rel}\,\partial \mc{C}_N$. {Moreover, $Y^* \coloneqq \lambda B - P_{\mc{C}_N}(\lambda B)$ is not a constant vector in $(\mathbb{R}^d)^N$, i.e., there does not exist $y \in \mathbb{R}^d$ such that $y_i^* = y$ for all $i \in \{ 1,\ldots,N\}$.}
\end{lemma}
\begin{proof}
Since  $ B \in \mr{rel}\,\partial \mc{C}_N$, there exists a vector $Y\in (\mathbb{R}^d)^N$, different from the constant vector, such that
\begin{equation}\label{eq:argmaxb}
B \in \mr{argmax}\{ \langle Y,X\rangle\,;\, X\in \mc{C}_N\}\,.
\end{equation}
Furthermore, we can pick $Y$ parallel to the affine hull of $\mc{C}_N$ defined by equation \eqref{eq: 1 is normal to C_N}, i.e.\ such that $\sum_i v_i y_i =0$. Note that the constant vector
 \[\mr{bary}(\rho)_N \coloneqq (\mr{bary}(\rho),\ldots,\mr{bary}(\rho)) \in (\mathbb{R}^d)^N\]
 is in the relative interior of $\mc{C}_N$, by the third point of Theorem \ref{th:boundary}. Therefore, because of \eqref{eq:argmaxb}, we have
\begin{equation}
\label{eq: B.Y > 0}    
\langle B, Y\rangle_{v} > 
\langle \mr{bary}(\rho)_N , Y\rangle_{v} = 
\sum_i v_i \langle \mr{bary}(\rho) , y_i \rangle = 0\,.
\end{equation}
This implies that $\lambda B \notin \mc{C}_N$ and $P_{\mc{C}_N}(\lambda B)$ belongs to the relative boundary of $\mc{C}_N$.
In fact, since 
$\langle \lambda B, Y\rangle_{v} > \langle B,Y\rangle_{v}$ by equation \eqref{eq: B.Y > 0}, 
then $\lambda B$ and $\mc{C}_N$ are separated by the supporting hyperplane containing $B$ given by
\[
\{ X \in (\mathbb{R}^d)^N\,;\, \langle X,Y\rangle_v = \langle B,Y\rangle_v\}\,.
\]
In addition, define $Y^*\coloneqq \lambda B - P_{\mc{C}_N}(\lambda B)$. By the classical characterisation of the orthogonal projection (e.g. \cite[Theorem 8.4]{BauschkeMoursi2023}),
\[
\langle Y^*, \mr{bary}(\rho)_N - P_{\mc{C}_N}(\lambda B)\rangle_v < 0\,,
\]
where the inequality is strict because $\mr{bary}(\rho)_N$ belongs to the relative interior of $\mc{C}_N$. This means that $Y^*$ cannot 
have
the form $y^*_i=y$ for all $i$, since otherwise this inner product would be zero
by equation \eqref{eq: 1 is normal to C_N}.
\end{proof}

As a direct consequence of Theorem \ref{th:interior} and Lemma \ref{lem:projection}, we obtain the following result:

\begin{proposition} \label{prop:projection} Suppose that $B\in \mc{B}^N_{\mr{iLag}}(v,\rho)$.
Then for any $\lambda>1$ sufficiently small,  $P_{\mc{C}_N}(\lambda B)\in \mc{B}^N_\mathrm{iLag}(v,\rho)$ and moreover $Y^* \coloneqq \lambda B - P_{\mc{C}_N}(\lambda B) \in (\mathbb{R}^d)^N \setminus \Delta_N$ is a vector of generators associated with $P_{\mc{C}_N}(\lambda B)$.
\end{proposition}

Proposition \ref{prop:projection} says that we can recover an irreducible Laguerre tessellation with barycenters arbitrary close to a given $B \in \mc{B}_{\mr{iLag}}^N$ via an orthogonal projection. 
On the other hand, in the general case where $B \in \mc{B}_{\mr{Lag}}^N$ or $B \in \mc{B}_{\mr{hLag}}^N$, we cannot guarantee that $Y^*\notin \Delta_N$. Even when $B \in \mc{B}_{\mr{iLag}}^N$, Proposition \ref{prop:projection}
does not provide a quantitative estimate on how small $\lambda$ should be in order to obtain a set of generators $Y^*\notin \Delta_N$. However,  numerically, the approach seems robust with respect to the choice of $\lambda$, in the sense that it generally provides a set of distinct generators even for $\lambda \gg 1$. Note also that as $\lambda \rightarrow 1$, $B^* \rightarrow B$ and by stability of optimal transport maps (see, e.g., Theorem 5.20 in \cite{villani2009optimal}) one can deduce that the Laguerre tessellation associated with $B^*$ converges to the one associated with $B$ (i.e., corresponding cells converge in measure). However it is not straightforward to quantify this convergence in terms of $\|B - B^*\|_v$.

We conclude the section with two further reinterpretations of the projection $P_{\mc{C}_N}(\lambda B)$. First, by the convex duality in Lemma \ref{lem:dual}, if $P_{\mc{C}_N}(\lambda B)$ is an exposed point that does not belong to any face of dimension strictly larger than zero, its exposing normal is
\[
Y = \lambda B - P_{\mc{C}_N}(\lambda B)\,.
\]
Moreover, $Y$ 
is the unique minimizer for the problem
\[
 \inf_{Y} \left\{\frac{\| Y\|^2_{v}}{2\lambda} - \langle Y,B\rangle_{v} + \frac{1}{\lambda}F_N(v,\rho;Y)\right\}\,.
\]
In other words, by introducing $\lambda >1$, we are weakening the problem’s strong convexity to avoid the trivial solution $Y=0$, which is the unique solution of the problem when $\lambda=1$ and $B \in \mc{B}_{\mr{Lag}}^N$. On the other hand, note that this is a strongly convex unconstrained problem, and therefore easier to solve than problem \eqref{eq:legF}, which is only convex and constrained. The downside is that here we are only computing an approximate solution. Remarkably, however, in numerical tests the projection approach generally produce good results even for relatively large values of $\lambda$ (see Section \ref{sec:numerical}).

The second interpretation is a consequence of Proposition \ref{prop:w2proj}. Specifically, $P_{\mc{C}_N}(\lambda X)$ can be related to the set of particle positions at time $t=\lambda/(\lambda-1)>1$ obtained by extending (in a specific way) the $W_2$-geodesic connecting $\rho$ at time $0$ to $\nu_N(v,X)$ at time $1$. Such an extension was introduced in \cite{gallouet2025metric} under the name of  \emph{metric extrapolation}, and  is defined via a minimization of a difference of squared Wasserstein distances, as explained in the following remark.

\begin{remark}[Relation with the metric Wasserstein extrapolation] 
Let us consider the problem of projecting $\lambda X$ onto 
$\mc{C}_N$ with $\lambda>1$, as in Proposition \ref{prop:projection}, but with $X\in (\mathbb{R}^d)^N$ arbitrary. By the change of variable $Y = (\lambda-1) Z$, the dual problem in Lemma \ref{lem:dual} becomes up to constant terms
\begin{equation}\label{eq:duallambda}
\inf_{Z {\in (\mathbb{R}^d)^N}} \left\{ \lambda (\lambda-1) \frac{\|Z-X\|_{v}^2}{2} -(\lambda-1) \frac{W^2_2(\nu_N(v,Z),\rho)}{2}
\right\}.\end{equation}
Set $t= \lambda/(\lambda-1)$. Then the infimum in \eqref{eq:duallambda} is larger than
\begin{equation}\label{eq:extrap}
\lambda \inf_{\mu {\in \mc{P}_2(\mathbb{R}^d)}} \left\{  \frac{W_2^2(\mu,\nu_N(v,X))}{2(t-1)} -\frac{W^2_2(\mu,\rho)}{2t}\
\right\}\,.\end{equation}
In fact, $\|Z-X\|_{v} \geq W_2(\nu_N(v,Z),\nu_N(v,X))$, and we get  problem  \eqref{eq:extrap} by replacing the minimization over $Z$ with a minimization over any probability measure $\mu$ with finite second moments. This is precisely the metric extrapolation problem, which was introduced in \cite{gallouet2025metric} to extend up to time $t>1$ the $W_2$-geodesic connecting $\rho$ at time $0$ to $\nu_N(v,X)$ at time $1$. It was proven in \cite{gallouet2025metric} that the solution to this problem is unique and 
has
the form $\mu^* = T_\# \nu_N(v,X)$
where $T$ is a continuous map. This implies that problem \eqref{eq:extrap} actually coincides with \eqref{eq:duallambda} and $\mu^*= \nu_N(v,Z^*)$ with $Z^*$ solving \eqref{eq:duallambda}.
\end{remark}

\section{Numerical methods and tests} \label{sec:numerical}
In this section we describe two numerical strategies to solve the projection problem considered in Section \ref{sec:projection}, and then apply them to the reconstruction and fitting of Laguerre tessellations.

\subsection{Subgradient descent}
Let us go back to the dual problem in Lemma \ref{lem:dual}, i.e.\
\begin{equation}\label{eq:dualG}
\inf_Y G_N(Y)\,, \quad \text{where}\quad
G_N(Y) \coloneqq \frac{\|Y\|_v^2}{2} -\langle Y,X\rangle_v +F_N(v,\rho; Y)\,.
\end{equation}
Note that minimizing $G_N$ is equivalent to evaluating the proximal operator of $F_N$ with respect to the weighted Euclidean norm $\| \cdot \|_v$; see \cite[Definition 25.1]{BauschkeMoursi2023}.

Given any $Y \in (\mathbb{R}^d)^N$, let $M \in\{1,\ldots ,N\}$ and $\tilde{Y}\in (\mathbb{R}^d)^M
\setminus
\Delta_M$ be such that $y_i =  \tilde{y}_{\sigma(i)}$ for all $i$, for some surjective map $\sigma:\{1, \ldots, N\} \rightarrow \{1, \ldots,M\}$. Define $w =(w_1, \ldots,w_M)$ by
\[
w_{j} = \sum_{i\in \sigma^{-1}(j)} v_i.
\]
Then, from the proof of Theorem \ref{th:boundary}, one can deduce that
\begin{equation}\label{eq:BY}
B(Y) \coloneqq (\mr{bary}_\rho(L^*_{\sigma(i)}(w,\tilde{Y})))_i \in \partial F_N(Y),
\end{equation}
where $\partial F_N$ denotes the subgradient of $F_N$ with respect to the metric $\langle \cdot,\cdot\rangle_v$. 

\begin{lemma}\label{lem:lip} Suppose $\rho \in \mc{P}_2(\mathbb{R}^d)$, i.e.\
\[
m_2^2(\rho) \coloneqq \int_{\mathbb{R}^d} |x|^2 \ed \rho(x) < +\infty.
\]
Then $F_N$ is $m_2(\rho)$-Lipschitz. Moreover, denoting by $Y^*$ the unique solution of
\eqref{eq:dualG}, $\|Y^*-X\|_v\leq m_2(\rho)$ and $G_N$ is $2m_2(\rho)$-Lipschitz on the set $\{Y\;;\; \|Y-X\|_v\leq m_2(\rho)\}$.
\end{lemma}
\begin{proof}
For any $Y,Z \in (\mathbb{R}^d)^N$, since $B(Z) \in \partial F_N(Z)$,
\[
\begin{aligned}
F_N(Z) -F_N(Y) &\leq \langle B(Z), Z-Y\rangle_v\\
&\leq \|B(Z)\|_v \|Z-Y\|_v \\&\leq m_2(\rho) \|Z-Y\|_v\,,
\end{aligned}
\]
where the last inequality is due to Jensen's inequality. Exchanging $Y$ and $Z$ we obtain that $F_N$ is $m_2(\rho)$-Lipschitz. Moreover, by Lemma \ref{lem:dual},
\[
\|Y^*-X\|_v = \| P_{\mc{C}_N}(X)\|_v \leq m_2(\rho)\,,
\]
where the last inequality follows from the definition of convex order and choosing as test function $\varphi:x\rightarrow |x|^2$. We conclude by observing that
\[
\partial G_N(Y) = Y-X + \partial F_N(Y)
\]
and using the triangle inequality.
\end{proof}

Given the considerations above, we can formulate a projected subgradient descent method with varying stepsize $\eta_k$ (see, e.g., Section 3.4.1 in \cite{bubeck2015convex}) as follows: given $Y_1 \in (\mathbb{R}^d)^N$, for all $k\geq 1$,
\begin{equation}\label{eq:sg}
\begin{aligned}
\textbf{Step 1. }&Z_k = Y_k - \eta_k (Y_k -X + B(Y_k))\,,\\
\textbf{Step 2. }&Y_{k+1}  = X + {\max\left(\frac{\|Z_k-X\|_v}{m_2(\rho)}, 1\right)}^{-1} (Z_k-X) \,.
\end{aligned}
\end{equation}

As a consequence of Lemma \ref{lem:lip}, the strong convexity of $G_N$, and 
Theorem 3.9 in \cite{bubeck2015convex} we have that the subgradient method defined by \eqref{eq:sg} converges with sublinear rate with an appropriate choice of step-size. More precisely:
\begin{lemma}\label{lem:convergencesg} Choosing $\eta_k = 2/(k+1)$, the iterates of algorithm \eqref{eq:sg}  
{satisfy}
for $k\geq 1$
\[
G_N\left( \sum_{j=1}^k \frac{2j}{k(k+1)} Y_j\right) - G_N(Y^*) \leq \frac{4 m_2^2(\rho)}{k+1}\,.
\]
\end{lemma}

\subsection{Frank-Wolfe algorithm} \label{sec:fw} The projection problem
\[
\inf_{B\in \mc{C}_N} J_N(B)\,, \quad J_N(B) \coloneqq\frac{1}{2} \|B-X\|^2_v\,,
\]
can be solved by applying directly the classical Frank-Wolfe algorithm \cite{frank1956algorithm}.
In our setting, this amounts to the following iterative scheme, which allows us to compute $\mc{P}_{\mc{C}_N}(X)$ from a given initial guess $B_1\in \mc{C}_N$: for all $k\geq 1$,
\begin{equation}\label{eq:fw}
\begin{aligned}
\text{\textbf{Step 1}. } & \text{Compute } C_{k}= B(X-B_k) \in \partial F_N(X-B_k),\\
\text{\textbf{Step 2}. } & \lambda_{k} = \arg \min \{ \|\lambda B_k+(1-\lambda) C_k-X\|^2_v \,;\, \lambda \in [0,1]\},\\
\text{\textbf{Step 3}. } & B_{k+1} = \lambda_k B_k+(1-\lambda_k) C_k\,.
\end{aligned}
\end{equation}

Note that in the first step $B(\cdot)$ is defined as in \eqref{eq:BY}.
The expression for $C_k$ is derived by minimizing over $\mc{C}_N$ the linearization of $J_N$ about $B_k$. 
As for the convergence of the algorithm, we start by observing that $\mr{diam}(\mc{C}_N) \leq 2 {m_2(\rho)}$. In fact, for any $X_1,X_2 \in \mc{C}_N$, by the triangle inequality
and Jensen's inequality,
\[
\|X_1-X_2\|_v^2 \leq 2 \|X_1\|_v^2 +2 \|X_2\|^2_v \leq 4 m_2^2(\rho)\,.
\]
Then, as a direct application of Theorem 13.14 in \cite{beck2017first}, we find:
\begin{lemma}\label{lem:convergencefw} The iterates of algorithm \eqref{eq:fw}  
{satisfy}
for $k\geq 2$
\[
J_N(B_k)
- J_N(P_{\mc{C}_N}(X)) \leq \frac{8 m_2^2(\rho)}{k-1}\,.
\]
\end{lemma}

\subsection{Numerical tests} We now present some numerical results that illustrate the behavior of our method for the reconstruction and fitting problems. More precisely, given a set of volumes $v$ and barycenters $B$, we reconstruct the  generators of the associated tessellation by setting
\[
Y = \lambda B - P_{\mc{C}_N}( \lambda B)\,,
\]
where $\lambda = t/(t-1)$ with $t>1$ is a regularization parameter. In practice, $Y$ is computed either directly via algorithm \eqref{eq:sg} or by setting $Y = X-B$ where $B$ is obtained via algorithm \eqref{eq:fw}. For the tests considered here, we found numerically that particles stayed sufficiently far from each other to allow for the computation of the optimal tessellation with $N$ distinct particles (which was realized using the open-source library \texttt{sd-ot}, available at \url{https://github.com/sd-ot}).
A Jupyter notebook containing the numerical experiments and code used in this paper is available at \url{https://github.com/andnatale/Laguerre-fitting/}.

\subsubsection{Reconstruction of Laguerre tessellations}
 We consider the problem of reconstructing a Laguerre tessellation given its cell volumes and barycenters. 
For all tests in this section, we use as reference measure $\rho = \mathbf{1}_{[0,1]^2}\ed x$.

Figure \ref{fig:reconstruction} shows the result obtained using algorithm \eqref{eq:sg} with $t=1250$ and $N=20$. 
The reconstruction obtained using the Frank-Wolfe scheme \eqref{eq:fw} is identical. Note that here the initialization of the subgradient scheme is a random distribution of generators (depicted in Figure \ref{fig:reconstruction}, center) and the one for the Frank-Wolfe scheme is the associated distribution of barycenters. The convergence of the two schemes is illustrated in Figures \ref{fig:convergence_reconstruction} and \ref{fig:convergencet_reconstruction}, both in terms of number of iterations and increasing the regularization parameter $t$. One can observe that both schemes appear to converge faster than 
predicted 
in Lemmas \ref{lem:convergencesg} and \ref{lem:convergencefw},
and that the Frank-Wolfe algorithm exhibits almost linear convergence, especially for small values of $t$.

The results corresponding to a larger test ($N=728$) are shown in Figure \ref{fig:reconstruction_cross},  \ref{fig:generators_reconstruction_cross},  \ref{fig:initial}, and \ref{fig:convergence_reconstruction_cross}. Note that while the reconstructed barycenters approximate well the data, the reconstructed generators  are very different (see Figure \ref{fig:generators_reconstruction_cross}) and the shape of specific Laguerre cells may differ considerably with respect to the true ones 
(see Figure \ref{fig:reconstruction_cross}). This is expected since the tessellation generating the data is not irreducible. As for the convergence of the schemes, we consider two possible initial conditions shown in Figure \ref{fig:initial}, the first being associated to a uniform distribution of generators, and the second corresponding the the choice $Y_1 = B$ and $B_1 = B$. Note that in this symmetric configuration this choice yields already a good approximation of the target tessellation. The behavior of the subgradient descent scheme is however unchanged, whereas the Frank-Wolfe scheme, even if starting with a much lower error, converges slowly at first and eventually performs similarly to the subgradient scheme.  

\begin{figure}
    \centering
    \includegraphics[scale=.7,trim = 0 15 0 0, clip = True]{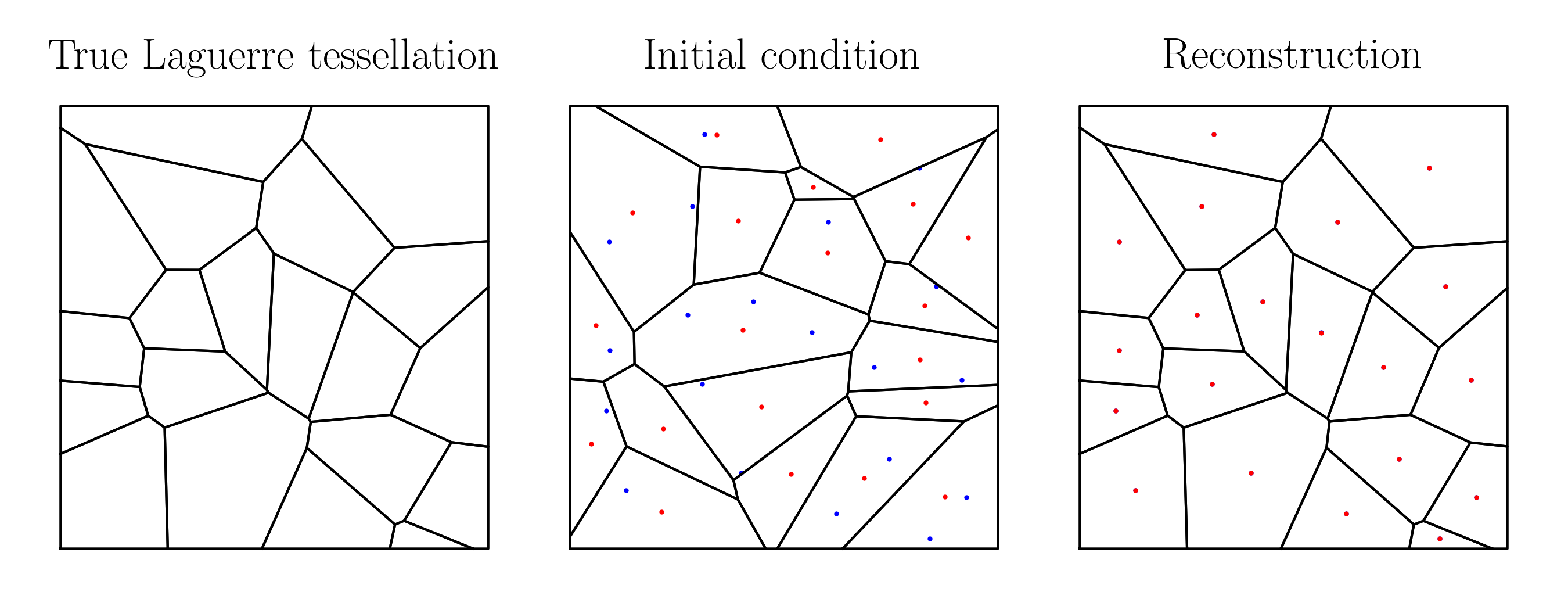}
    \caption{Reconstruction of a Laguerre tessellation, for $t=1250$ and the subgradient descent scheme. The barycenters of the true solutions are in blue and those of the reconstruction are in red. The reconstruction obtained with the Frank-Wolfe scheme is identical (not shown).}
    \label{fig:reconstruction}
\end{figure}

\begin{figure}
    \centering
    \includegraphics[width=0.45\linewidth]{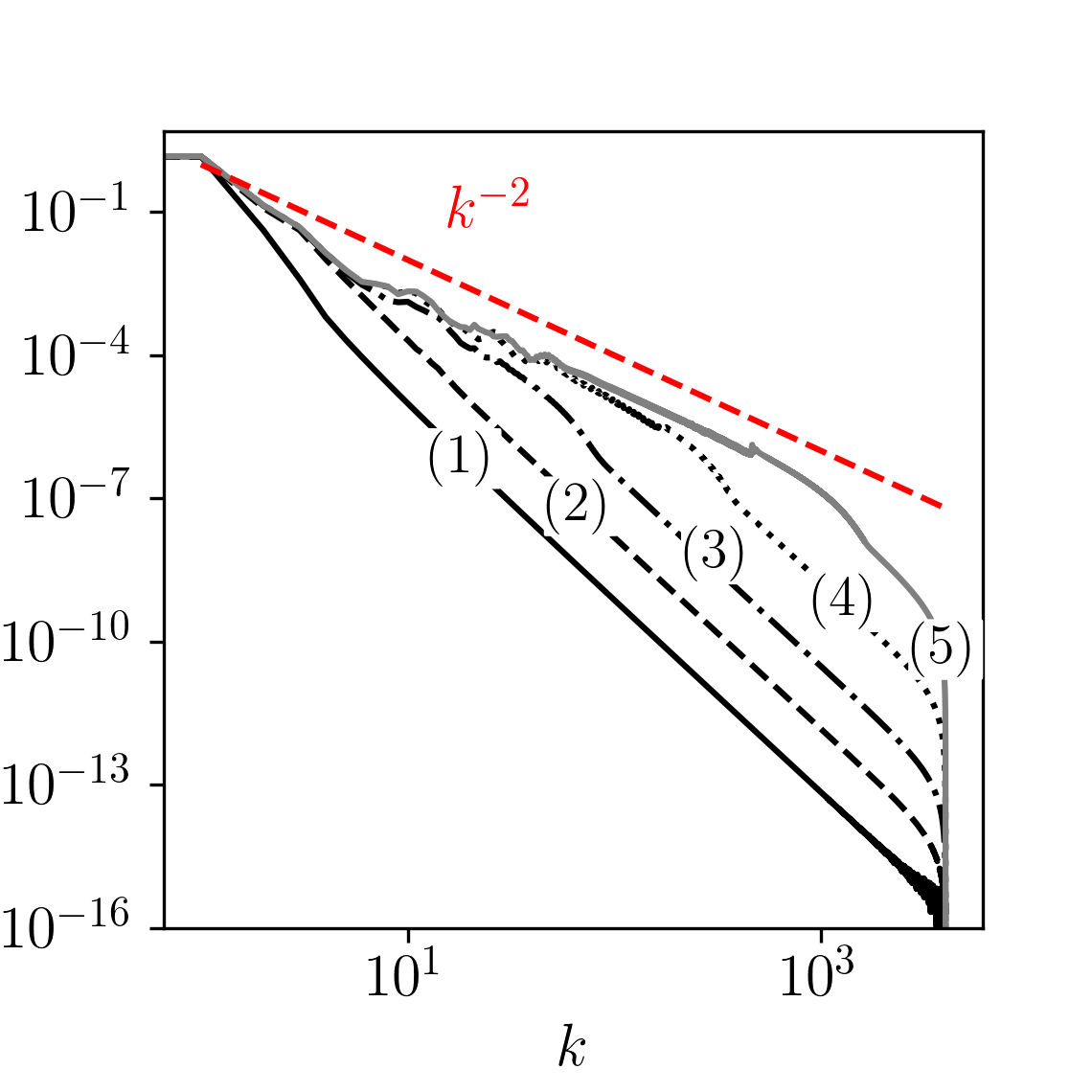}
    \hspace{-1em}
    \includegraphics[width=0.45
     \linewidth]{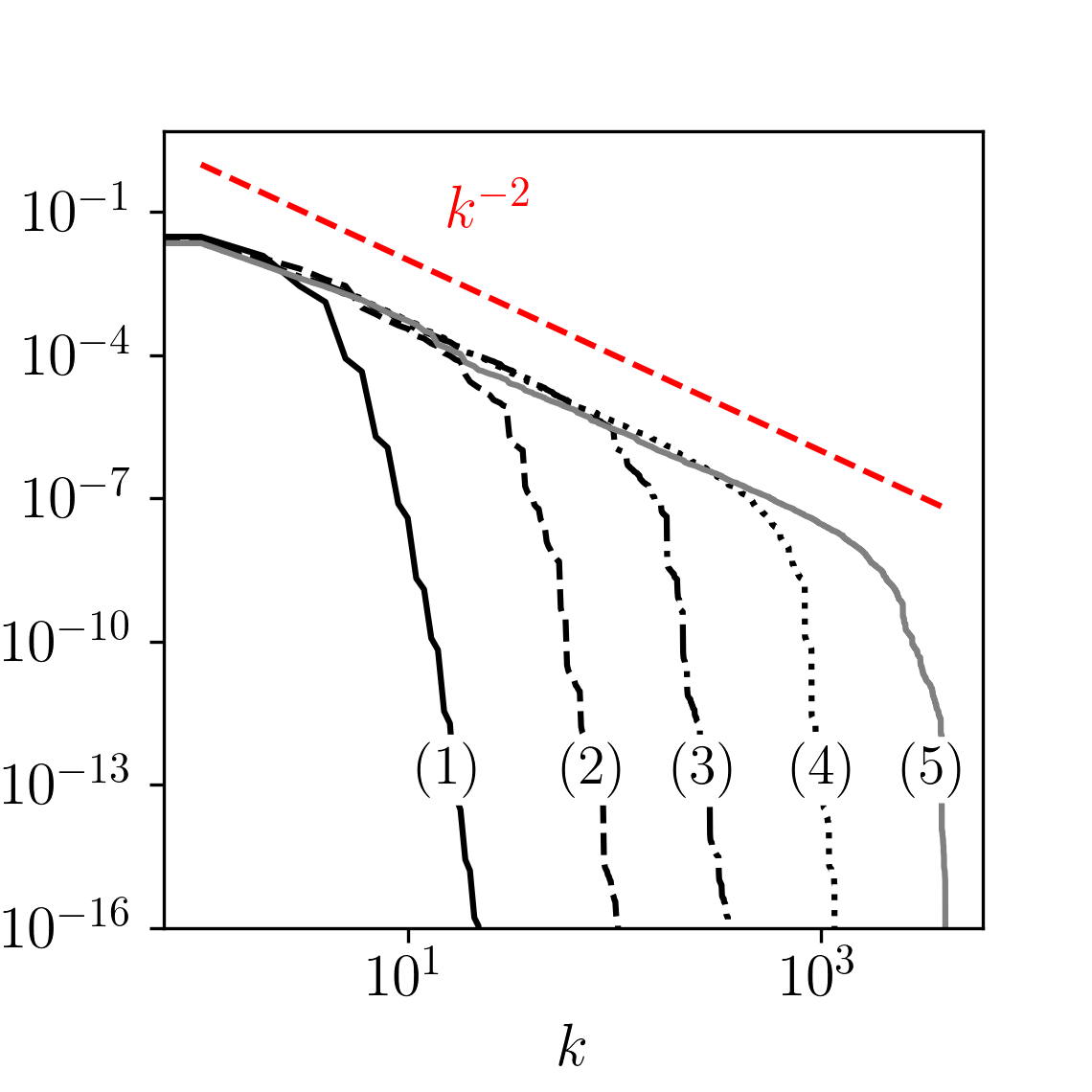}
    \caption{On the left: error $G_N(Y_k)- G_N(Y_K)$ 
    for the subgradient scheme, for $1\leq k \leq K$, for different values of $t$ (the curve $(i)$ corresponds to $t=2\cdot5^{i-1}$), and the data in Figure \ref{fig:reconstruction}. On the right: error 
    $J_N(B_k)- J_N(B_K)$ 
    for the Frank-Wolfe scheme. 
    }
\label{fig:convergence_reconstruction}
\end{figure}

\begin{figure}
    \includegraphics[width=0.45
     \linewidth]{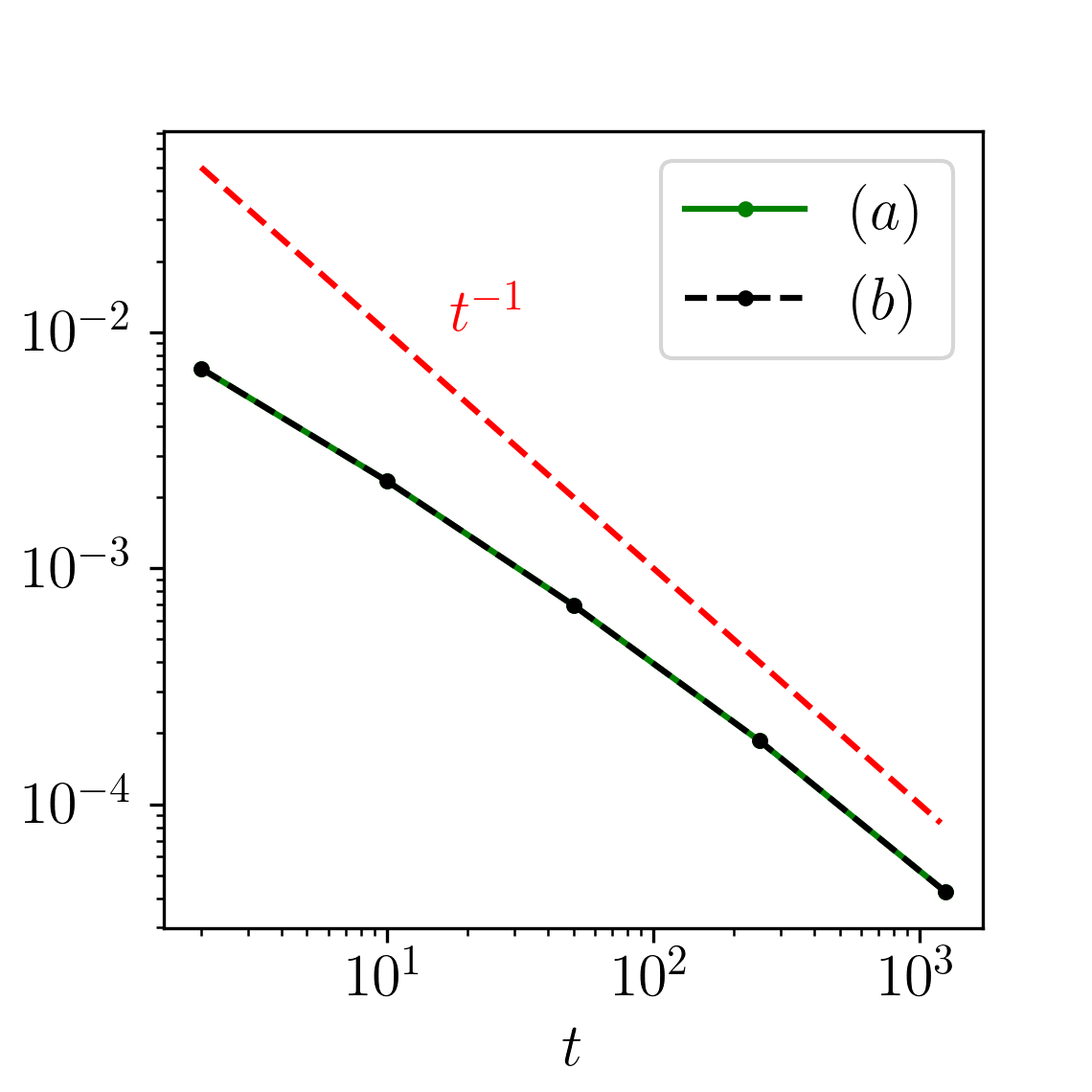}
\caption{Error 
of the
barycenter reconstruction, $\|B_K-B\|_v$, for different $t$ at the last iteration $K$ of 
(a)
the subgradient descent scheme
and 
(b) the
Frank-Wolfe 
scheme.
}
\label{fig:convergencet_reconstruction}
\end{figure}

\begin{figure}
    \centering
    \includegraphics[scale=.7, trim = 0 15 0 0, clip = True]{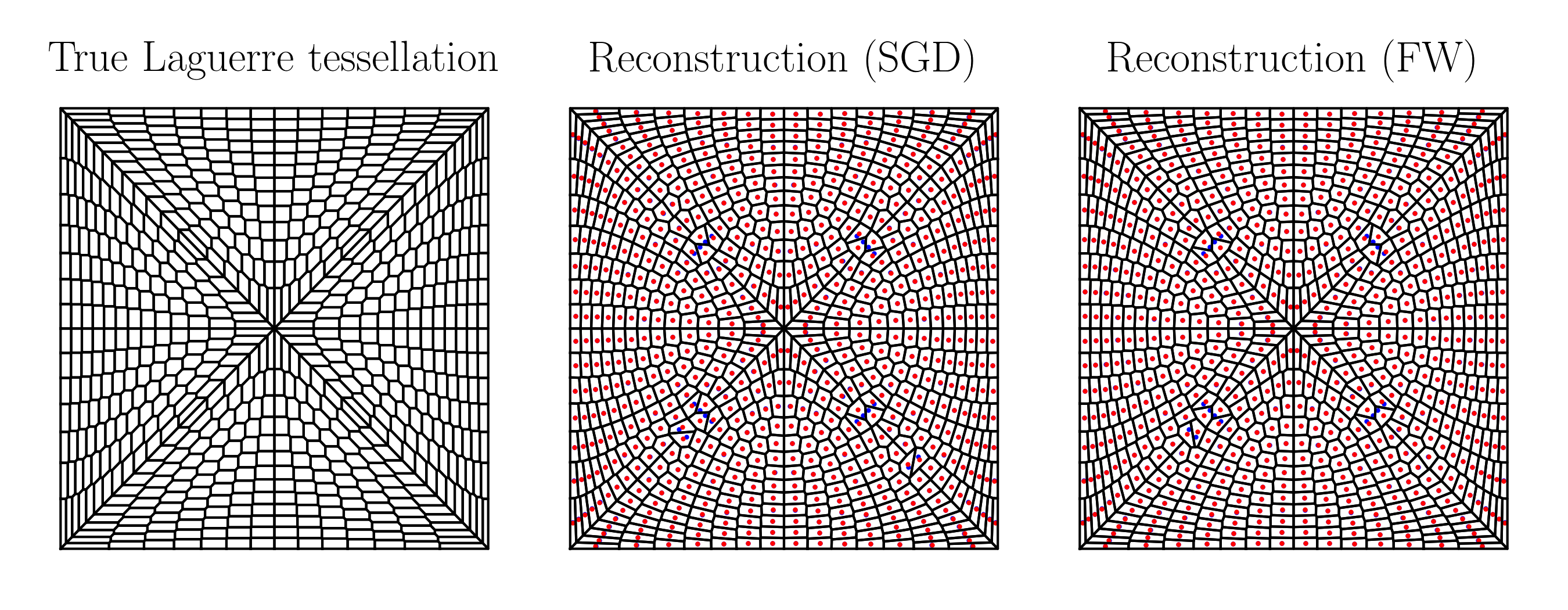}
    \caption{Reconstruction of a Laguerre tessellation, for $t=100$. The barycenters of the true solution are in blue and those of the reconstruction are in red (subgradient descent, center, and Frank-Wolfe, right).}
    \label{fig:reconstruction_cross}
\end{figure}

\begin{figure}
    \centering
    \includegraphics[scale=.7, trim = 0 35 0 0, clip = True]{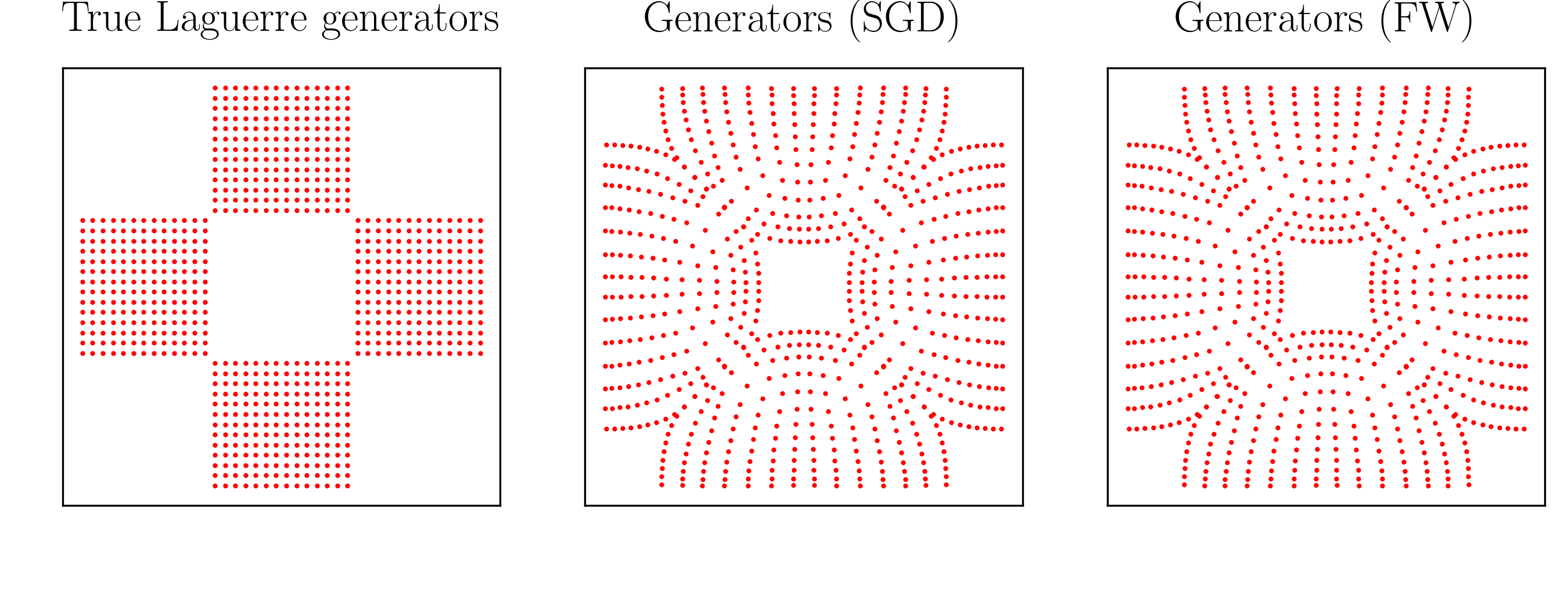}
    \caption{Generators of true Laguerre tessellation in Figure \ref{fig:reconstruction_cross} (left) and generators for reconstruction for $t=100$ (subgradient descent, center, and Frank-Wolfe, right).}
    \label{fig:generators_reconstruction_cross}
\end{figure}

\begin{figure}
    \centering
    \includegraphics[scale=.7, trim = 0 15 0 0, clip = True]{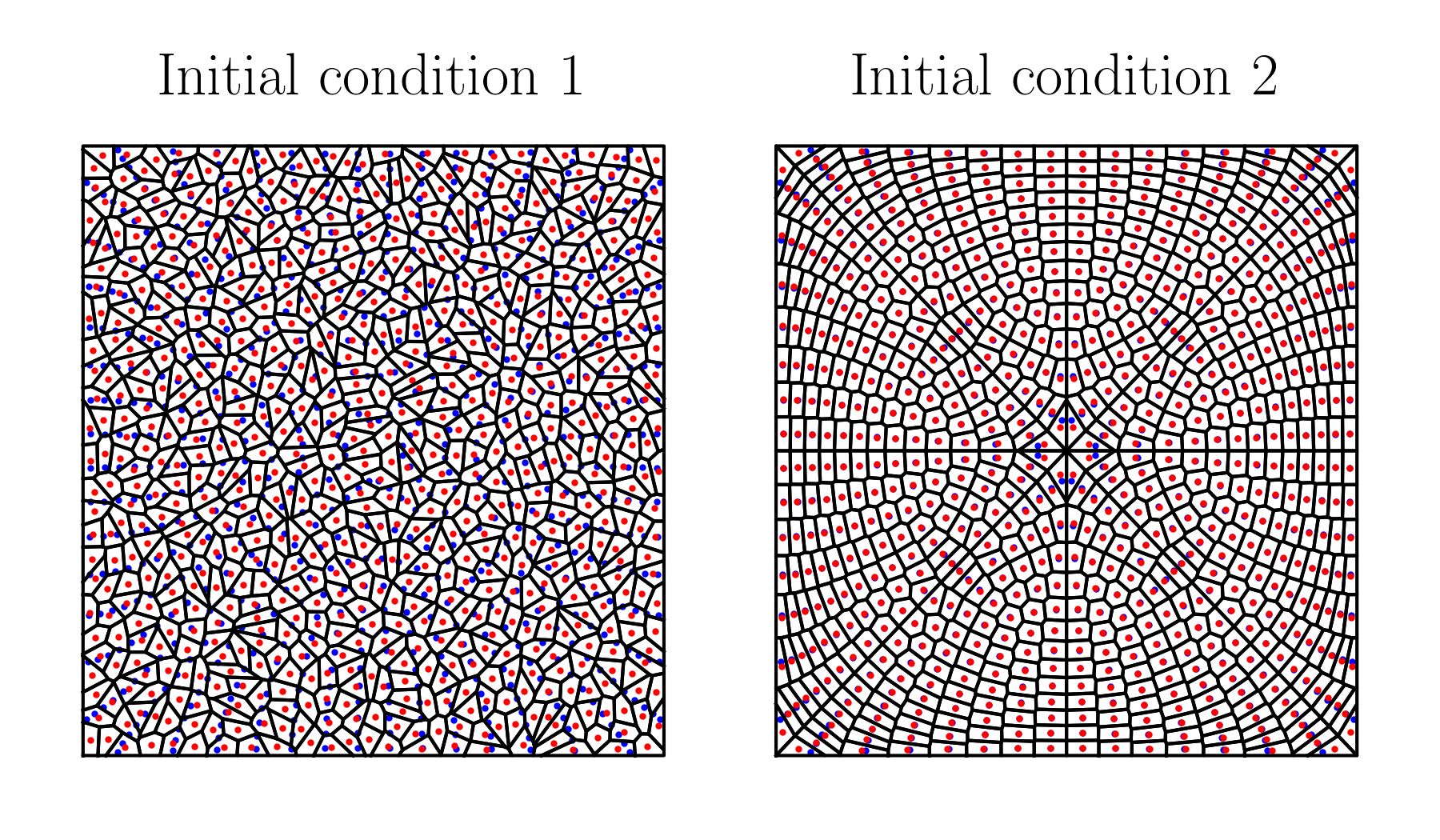}
    \caption{Initial conditions for the reconstruction problem in Figures \ref{fig:reconstruction_cross} and \ref{fig:generators_reconstruction_cross}. The barycenters of the true solutions are in blue and those of the reconstruction are in red.}
    \label{fig:initial}
\end{figure}

\begin{figure}
    \centering
    \includegraphics[width=0.45\linewidth]{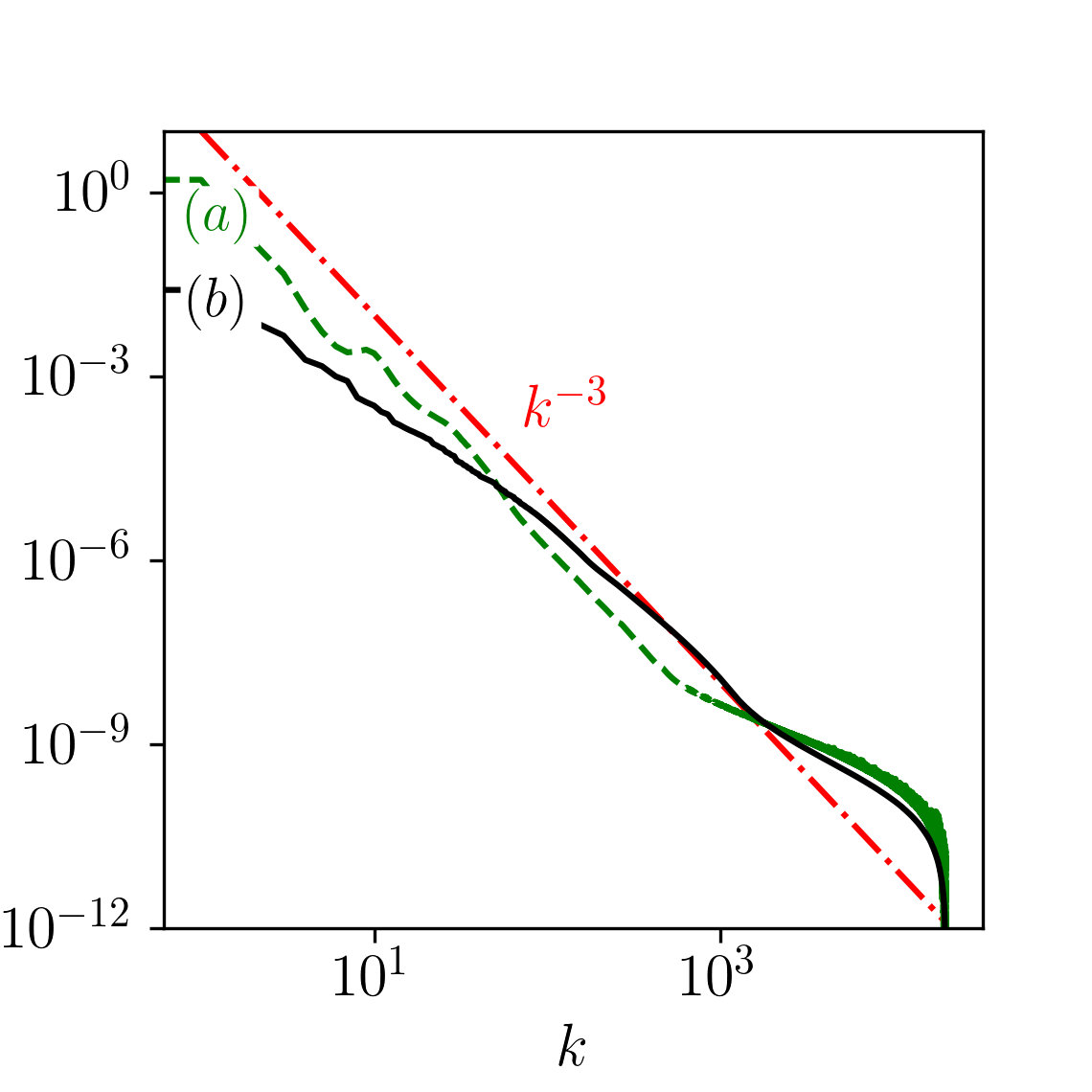}  \hspace{-1em}
    \includegraphics[width=0.45\linewidth]{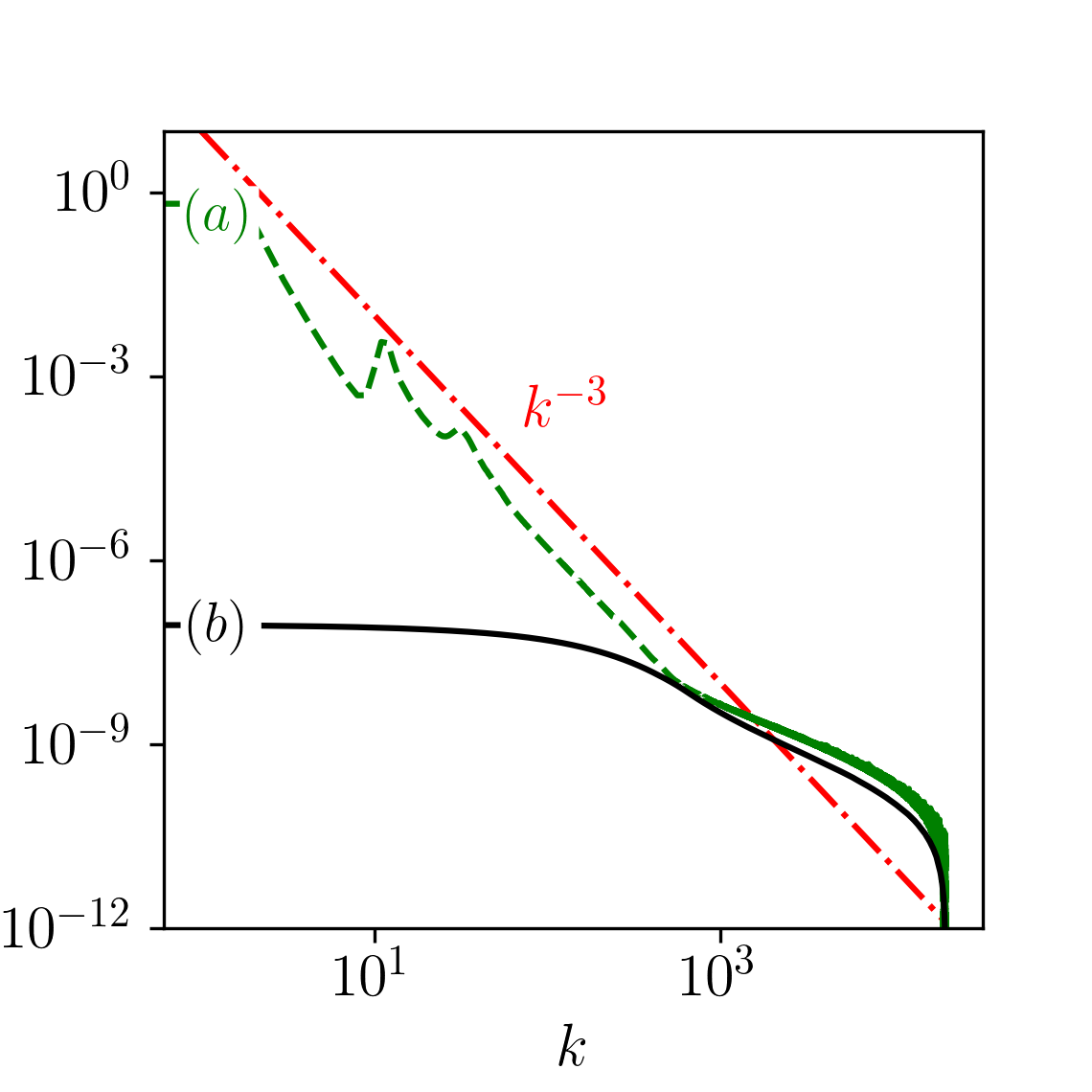}
    \caption{Convergence of the scheme for the data represented in Figure \ref{fig:reconstruction_cross} and $t=100$, shown in terms of $G_N(Y_k)- G_N(Y_K)$ 
    for the subgradient scheme $(a)$ and $J_N(B_k)- J_N(B_K)$ for the Frank-Wolfe scheme $(b)$, for $1\leq k \leq K$, and for the initial condition 1 (left) and 2 (right).}

    \label{fig:convergence_reconstruction_cross}
\end{figure}

\subsubsection{Fitting a Laguerre tessellation to data} 
\label{subsubsec:EBSD}
As in \cite{bourne2024inverting}, we apply our method to fit a Laguerre tessellation to an electron backscatter diffraction (EBSD) image of a single-phase steel (provided by Tata Steel Netherlands),
shown in Figure \ref{fig:fitting}. The pixels
are colored according to their crystallographic orientation, and the regions where this is constant are called grains. The image we consider has $N=243$ grains. 

We used the areas $v$ and the barycenters $B$ of the grains as data in our method to generate a vector of generators $Y$.
Note that since by construction $B\in \mc{C}_N$, we do need to include a regularization also in this  case. For the tests presented here we set $t=10$.
We used as reference measure $\rho = \mathbf{1}_\Omega \ed x$ where $\Omega = [0,252.25]^2$ is the image domain (measured in microns).
The results are shown in Figures \ref{fig:fitting} and \ref{fig:fitting_convergence}. For this test, both schemes are initialized using with the target barycenters, i.e.\ $Y_1= B$ and $B_1 = B$. As in the previous example, after the first iterations, the two schemes perform similarly.

\begin{figure}
    \centering
    \includegraphics[width=0.42\linewidth, trim = -22 -8  -22 -22, clip=True]{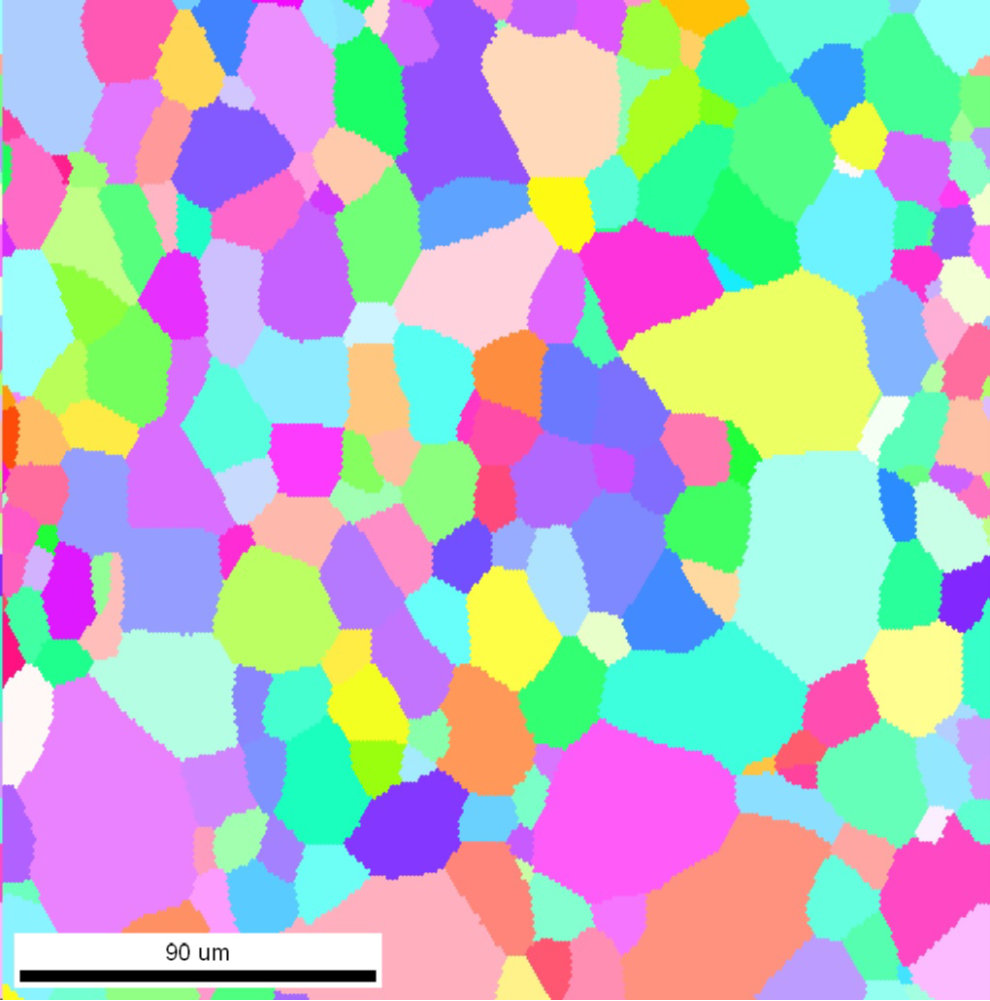}
    \includegraphics[width=0.42\linewidth,trim = 0 10  0 0, clip=True]{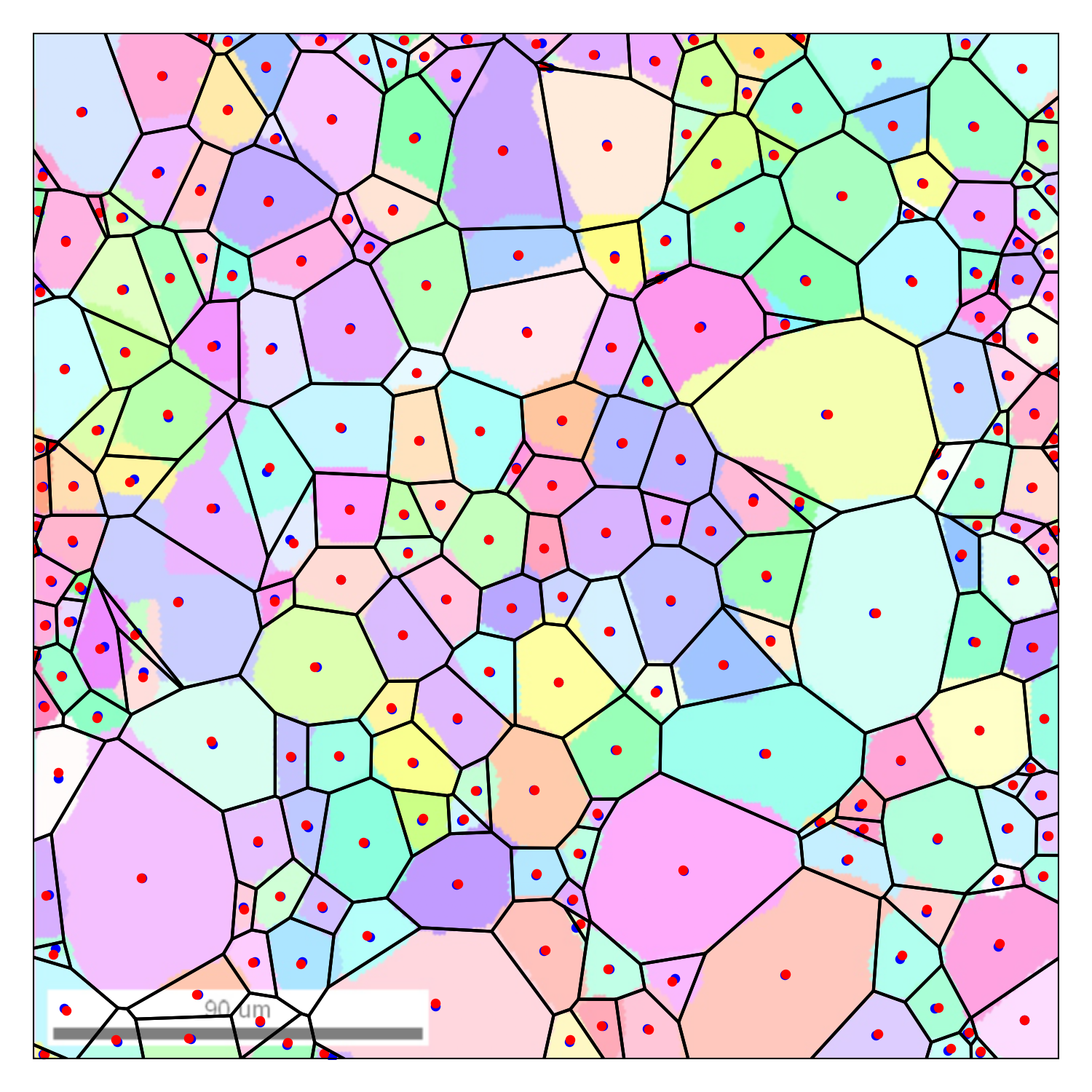}

    \hspace{-2em}
    \caption{
   Left: An EBSD image of steel. Right: A Laguerre tessellation fitted to the areas and barycenters of the grains in the EBSD image, overlaid over the image. The barycenters of the grains are in blue, the barycenters of the fitted Laguerre cells are in red. This was
    computed with $t=10$ and the subgradient descent scheme. The fitted tessellation computed with the Frank-Wolfe scheme is almost identical (not shown). 
    }
      \label{fig:fitting}
\end{figure}

    \begin{figure}
        \centering
\includegraphics[width=0.45\linewidth]{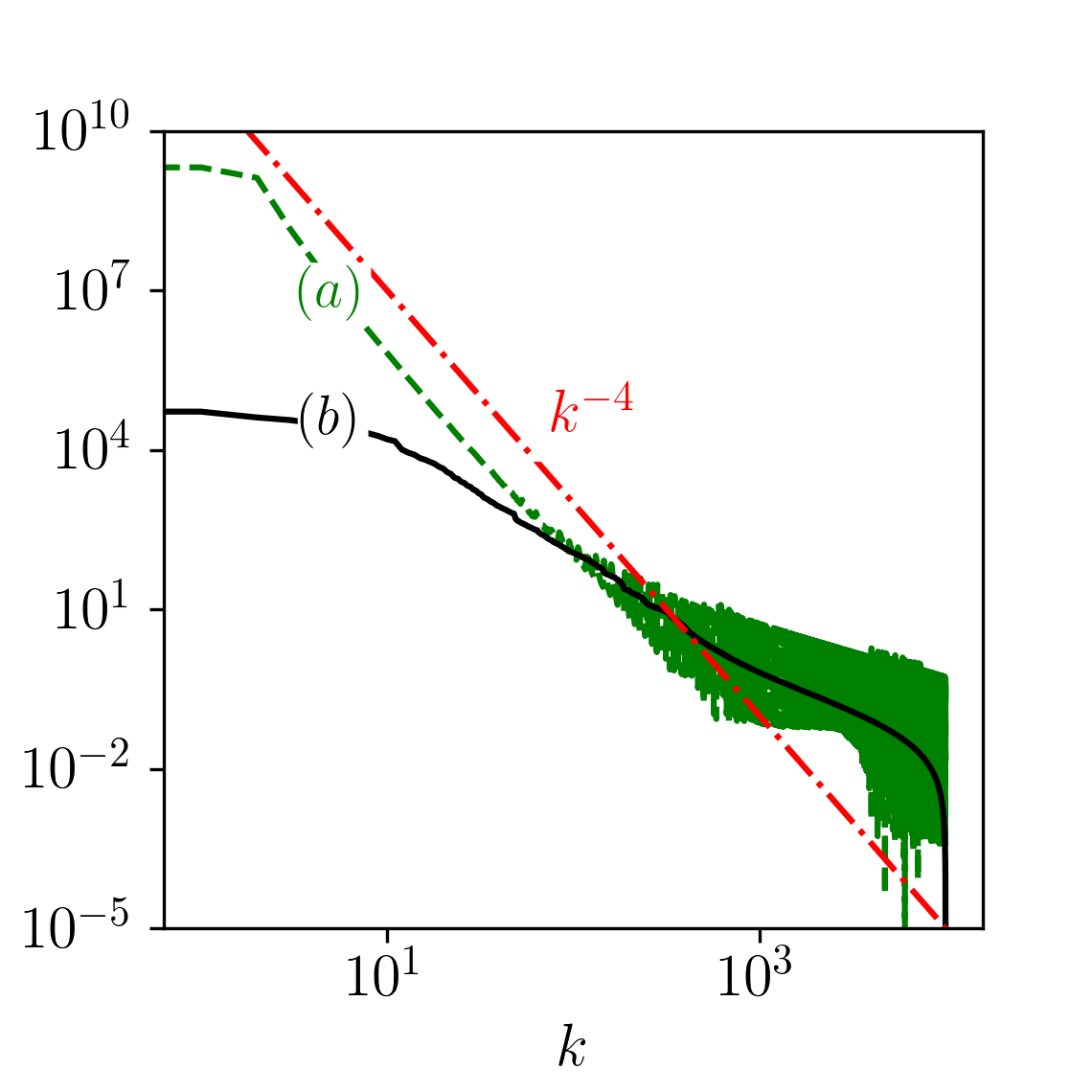}
\caption{Convergence of the scheme for the data extracted from the EBSD image in Figure \ref{fig:fitting} and $t=10$, shown in terms of $G_N(Y_k)- G_N(Y_K)$ 
and $J_N(B_k)- J_N(B_K)$ for the subgradient descent scheme $(a)$ and Frank-Wolfe scheme $(b)$,
respectively, and for $1\leq k \leq K$. } \label{fig:fitting_convergence}
\end{figure}

\section*{Acknowledgements}
The authors thank Irène Waldspurger for her valuable feedback on an earlier version of this manuscript. 
This work  was partly supported by the Labex CEMPI (ANR-11-LABX-0007-01). DPB would like to thank the UK Engineering and Physical Sciences Research Council (EPSRC) for financial support via the grant EP/V00204X/1. AN acknowledges funding by the Agence Nationale de la Recherche (ANR), project ANR-25-CE40-3242-01.  
\bibliographystyle{plain}
\bibliography{refs}

\end{document}